\documentclass[12pt]{article}
\usepackage{amsfonts}
\usepackage{graphicx}
\usepackage{latexsym,amsmath,color,esint}
\usepackage{natbib}

\def\esp{\mathbb{E}}

\def\1{\mathbb{I}}

\newcounter{thm}[section]
\newcounter{appen}[section]

\newtheorem{theor}[thm]{Theorem}

\newtheorem{lem}[thm]{Lemma}
\newtheorem{lem_a}[appen]{Lemma}

\newenvironment{proof}[1][Proof]{\noindent \textbf{#1.}
}{\rule{0.5em}{0.5em}}
\newtheorem{hp}{Assumption}

\begin{document}

\title{Testing for the significance of functional covariates in  regression models}
\author{Samuel Maistre\footnote{Corresponding author. CREST (Ensai) \& IRMAR (UEB), France; samuel.maistre@ensai.fr} \quad and \quad
Valentin Patilea\footnote{CREST (Ensai) \& IRMAR (UEB), France; patilea@ensai.fr}}
\date{}

\maketitle

\qquad

\begin{abstract}
Regression models with a response variable taking values in a Hilbert space and hybrid covariates are considered. This means two sets of regressors are allowed, one of finite dimension and a second one functional with values in a Hilbert space. The problem we address is the test of the effect of the functional covariates. This problem occurs for instance when checking the goodness-of-fit of some regression models for functional data. The significance test for functional regressors in nonparametric regression with hybrid covariates and scalar or functional responses is another example where the core problem is the test on the effect of functional covariates.
We propose a new test based on kernel smoothing. The test statistic is asymptotically standard normal under the null hypothesis provided the smoothing parameter tends to zero at a suitable rate. The one-sided test is consistent against any fixed alternative and detects local alternatives \`a la Pitman approaching the null hypothesis. In particular we show that neither the dimension of the outcome nor the dimension of the functional covariates influences the theoretical power of the test against such local alternatives. Simulation experiments and a real data application illustrate the performance of the new test with finite samples.

\smallskip

\end{abstract}

\newpage

\section{Introduction}\label{intro}

Let $(\mathcal{H}_1, \langle\cdot,\cdot\rangle_{\mathcal{H}_1})$ and $(\mathcal{H}_2, \langle\cdot,\cdot\rangle_{\mathcal{H}_2})$ denote two possibly different Hilbert spaces. The main examples of Hilbert spaces we have in mind are $\mathbb{R}^p,$ for some $p\geq 1,$ and $L^2[0,1],$ the space of squared integrable real-valued functions defined on the unit interval.

Consider the random variables $U\in\mathcal{H}_1$ and $W\in\mathcal{H}_2$ and let
$Z$ be a column random vector in $\mathbb{R}^{q},$
$q\geq 0.$ By convention, $q=0$ means that $Z$ is a constant. Let $(U_{i},Z_i,W_i)$, $1\leq i \leq n,$ denote a sample of independent copies of $(U,Z,W)$. The statistical problem we consider is the test of the hypothesis
\begin{equation}
\mathbb{E}[U\mid Z,W]=0\qquad \text{a.s.},  \label{quang}
\end{equation}
against a general alternative like $\mathbb{P}(\mathbb{E}[U\mid Z,W]=0)<1.$ This type of problem occurs in many model check problems.



Consider the random variables $U\in \mathcal{H}_1$, $\widetilde W\in\mathcal{H}_2.$ For illustration, suppose that $U$ is centered. Consider the problem of \emph{testing the effect of the functional variable $\widetilde W$,} that is testing the condition  $\mathbb{E}[U\mid \widetilde W] = 0.$ \cite{Patilea2012b} proposed a test procedure  based on projections into finite dimension subspaces of $\mathcal{H}_2$. Their test statistic is somehow related to a Kolmogorov-Smirnov statistic in a finite dimension space with the dimension growing with the sample size. Here we propose an alternative route that avoids optimization in high dimension.
Let $Z=\langle \widetilde W, \phi_1 \rangle_{\mathcal{H}_2}$ where $\phi_1$ is an element of an orthonormal basis of  $ \mathcal{H}_2.$ Suppose that $Z$ admits a density with respect to the Lebesgue measure. The basis  of  $ \mathcal{H}_2$ could be the one given by the functional principal components which in general has to be estimated from the data. In such a case, the sample of  $Z_i'$s has to estimated too. Let  $W= \widetilde W - \langle \widetilde W, \phi_1 \rangle_{\mathcal{H}_2} \phi_1.$ Then, testing  $\mathbb{E}[U\mid \widetilde W] = 0$ is nothing but testing condition (\ref{quang}).


\cite{Aneiros-Perez2006} introduced the \emph{semiparametric functional partially linear models} as  an extension of the partially linear model to functional data. Such model writes as
$$
Y = Z^\top \beta  + m(W) + U, \quad \text{with}\quad \mathbb{E}[U\mid Z,W]=0\;\; a.s.,
$$
where $Y$ is a scalar response and $Z$ is a $q-$dimension  vector of random covariates,  $W$ is a random variable taking values in a functional space, typically $L^2[0,1].$  The  column vector of $q$ coefficients $\beta$ and the function  $m(\cdot)$ have to be estimated. Before estimating  $m(\cdot)$ nonparametrically, one should first check the significance of the variable $W$ which means exactly testing condition  (\ref{quang}).
In this example, the variable $U$ is not observed and the sample $U_1,\cdots,U_n$ could be estimated by the residuals of  the linear fit of $Y$ given  $Z.$ The estimation error for the sample of $U$ is of rate $O_{\mathbb{P}}(n^{-1/2})$ and could be easily proved to be negligible for our test.


Other examples of regression model checks that lead to a problem like (\ref{quang}) are the functional linear regression with scalar or functional responses, quadratic functional regression, generalized functional regression, \emph{etc.} See for instance \cite{Horvath2012} for a recent panorama on the functional regression models. In such situations one has to estimate the sample $U_1,\cdots,U_n$ from the functional regression model considered. The estimation error is in general larger than the parametric rate $O_{\mathbb{P}}(n^{-1/2}),$ but one can still show that, under reasonable conditions, it remains negligible for the test purposes. See \cite{Patilea2012b} for a related framework.

Another example, related to the problem of testing the effect of a functional variable, is the \emph{variable selection in functional nonparametric regression with functional responses}. Regression models for functional responses are now widely used, see for instance \cite{Faraway1997}. Two situations were studied: finite and infinite dimension covariates; see  \cite{Ramsay2005}, \cite{Ferraty2011},  \cite{Ferraty2012}. Consider the hybrid case with both finite and infinite dimension covariates. An important question is the significance of the functional covariates. In a more formal way, let  $Y\in \mathcal{H}_1$ be the response and let $Z\in \mathbb{R}^q$ and $W\in\mathcal{H}_2$ denote the covariates. Then the problem is to test the equality
$$
\mathbb{E}[Y\mid Z, W] = \mathbb{E}[Y\mid Z].
$$
Let $U= Y- \mathbb{E}[Y\mid Z].$ Then the problem becomes to test whether
$\mathbb{E}[ U \mid Z, W]=0$ almost surely, that is the condition  (\ref{quang}). Again the sample of the variable $U$ is not observed and has to be estimated by the residuals of the nonparametric regression of  $Y$ given $Z.$ See also \cite{Lavergne2014} for a related procedure.

As a last example where a condition like (\ref{quang}) occurs consider the problem of testing the independence between a random variable $Y$ and a functional spaced valued variable $\widetilde W.$ Without loss of generality, one could suppose that $Y$ takes values in the unit interval.
Define $U(t) = \mathbf{1}\{Y\leq t  \} - \mathbb{P}(Y\leq t),$ $t\in[0,1],$ that  is  centered and belongs to $L^2[0,1]$. The independence between $Y$ and $\widetilde W$ is equivalent to the condition $\mathbb{E}[U\mid \widetilde W] = 0.$ Conditional independence of $Y$ and a functional random variable given some finite random vector $Z$ could be also tested. It suffices to define $U(t)$ by centering with the conditional probability of the event $\{Y\leq t  \}$ given $Z$ and to check a condition like (\ref{quang}).

To our best knowledge the statistical problem we address in this work was very little investigated in full generality.  \cite{Chiou2007} and \cite{Kokoszka2008} investigated the problem of goodness-of-fit with functional responses.
\cite{Chiou2007} considered  plots of
functional principal components (FPC) scores of the response and the covariate. They
also used residuals versus fitted values FPC scores plots. However, such two dimension
plots could not capture all types of effects of the covariate on the response. \cite{Kokoszka2008} used the response and covariate FPC scores to build a test statistic with $\chi^2$ distribution under the null hypothesis of no linear effect. See also the textbook \cite{Horvath2012}. Again, by construction, such  tests cannot detect any nonlinear alternative.
The goodness-of-fit or no-effect against nonparametric alternatives has been recently
explored in functional data context. In the case of scalar response, \cite{Delsol2011} proposed a testing procedure adapted from the approach of \cite{Hardle1993}. Their procedure involves smoothing in the functional space and requires
quite restrictive conditions. \cite{Patilea2012a} and \cite{Garcia-Portugues2012} proposed alternative nonparametric
goodness-of-fit tests for scalar response and functional covariate using projections
of the covariate. \cite{Patilea2012b} extended the idea to functional responses and seems to be the only contribution allowing for functional responses. Such projection-based methods are less restrictive and perform well in applications. However, they require a search for the most suitable projection and this may involve optimization in high dimension.

The paper is organized as follows. In section \ref{sec_met} we introduce our testing approach, while in section \ref{ass_th_sec} we provide the  asymptotic analysis. The asymptotically standard normal critical values and the consistency of the test are derived. The application to goodness-of-fit tests of functional data models is discussed.
The extension to the case of estimated covariates is presented in section \ref{est_basis}. This allows in particular for an estimated basis in the infinite-dimensional space of the functional covariate.
Section \ref{emp_evid} presents some empirical evidence on the performances of our test and comparisons with existing procedures. The proofs and some technical lemmas are relegated to the Appendix.

\setcounter{equation}{0}
\section{The method}\label{sec_met}

Let us first introduce some notation. Let $\{\phi_1,\phi_2,\cdots  \}$ be some orthonormal basis of  $\mathcal{H}_2$ that for the moment is supposed to be fixed. In section \ref{est_basis} we consider the case of a data-driven basis.  For simplicity and without any loss of generality in the following, assume hereafter that $\mathbb{E}(W)=0.$ Then we can decompose $W=\sum_{k\geq 1} \langle W,\phi_k\rangle_{\mathcal{H}_2}\phi_k$ and the norm of $W$ satisfies the relationship $\|W\|^2_{\mathcal{H}_2}=\sum_{k\geq 1} \langle W,\phi_k\rangle_{\mathcal{H}_2}^2.$
Let us note that
$$
\mathbb{E}(U\mid Z,W)=\mathbb{E}(U\mid Z, \langle W,\phi_1\rangle_{\mathcal{H}_2},\langle W,\phi_2\rangle_{\mathcal{H}_2},\cdots).
$$
Next, for any positive integer $p,$ let
$$W_{\underline{p}}=(\langle W,\phi_1\rangle_{\mathcal{H}_2},\cdots, \langle W,\phi_p\rangle_{\mathcal{H}_2} )^\top.$$
For a function  $l$, let $\mathcal{F}[l]$ denote the Fourier Transform of  $l$. Let  $K$ be a multivariate kernel defined on $\mathbb{R}^{q}$ such that
$\mathcal{F}[K]>0$
and $\varphi(s)=\exp(-\|s\|^2/2),$  $\forall s\in\mathbb{R}^p,$ where here $\|\cdot\|$ is the Euclidean norm in $\mathbb{R}^p.$
Many univariate kernels satisfy the positive Fourier Transform condition, for instance the  gaussian, triangle, Student and  logistic densities. To obtain a multivariate kernel with positive Fourier Transform it suffice to consider a multiplicative kernel with positive Fourier Transform univariate kernels.

\subsection{The idea behind the testing method}

The new procedure proposed below is motivated by the following facts. First, if $(U_1,Z_1,W_1)$ and $(U_2,Z_2,W_2)$ are independent copies of $(U,Z,W),$  for any positive function $\omega (\cdot)$ and any  $h>0$ and $p$ positive integer, by the Inverse Fourier Transform formula,
\begin{multline}
I_p(h) =  \mathbb{E}\left[\langle U_{1}, U_{2}\rangle_{\mathcal{H}_1}\omega(Z_{1})\omega(Z_{2})h^{-q}K((Z_{1}-Z_{2})/h)\varphi(W_{1,\underline{p}}-W_{2,\underline{p}})\right]\nonumber \\
 =  \mathbb{E}\left[\langle U_{1}, U_{2}\rangle_{\mathcal{H}_1}\omega(Z_{1})\omega(Z_{2})\int_{\mathbb{R}^{q}}e^{2\pi it^{\top}(Z_{1}-Z_{2})}\mathcal{F}[K](th)dt\right.\\
 \left.\times \int_{\mathbb{R}^p}e^{2\pi is^\top(W_{1,\underline{p}}-W_{2,\underline{p}})
 }\mathcal{F}[\varphi](s)ds\right]\nonumber \\
 =  \int_{\mathbb{R}^{q}}\int_{\mathbb{R}^{p}}\left\|\mathbb{E}\left[\mathbb{E}[U\mid Z,W_{\underline{p}}]\omega(Z)e^{-i\{t^{\top}Z +  s^\top W_{\underline{p}} \}}\right]\right\|_{\mathcal{H}_1}^{2}\mathcal{F}[K](th)\mathcal{F}[\varphi](s)dtds.
\end{multline}
By the properties of the Fourier Transform and the conditions  $\mathcal{F}[\varphi], \mathcal{F}[K]>0$ (and $\omega>0$), for any $h>0$ and $p$  the real number $I_p(h)$ is nonnegative and
\[
\mathbb{E}(U\mid Z,W_{\underline{p}\;})=0\;\; a.s.\;\;\Leftrightarrow\;\; I_p(h)=0.
\]
Second, by a martingale convergence argument with respect to $p$, it follows that
$$
\mathbb{E}(U\mid Z,W)=0\;\; a.s.\;\;\Leftrightarrow\;\;
\mathbb{E}(U\mid Z,W_{\underline{p}\;})=0\;\; a.s. \quad\forall p\in\{1,2,\cdots\}.
$$
These intuitions are formalized in  the following fundamental lemma, up to some technical modification. In the following  $a_1,a_2,\cdots$ is a fixed sequence of positive real numbers. For any sequences $u=\{u_k\}_{k\geq 1},$ $v = \{v_k\}_{k\geq 1},$  let
\begin{equation}\label{inner_a}
\langle u, v\rangle_{\mathcal{A}} = \sum_{k\geq 1} a_k u_k v_k \quad \text{and} \quad \|u\|^2_{\mathcal{A}} =
\sum_{k\geq 1} a_k u_k^2,
\end{equation}
whenever the series converge.

\begin{lem}\label{lemma_fund}
Assume that $\omega(\cdot) >0,$ $\mathcal{F}[K]>0,$ $\mathcal{F}[K]$ is integrable and $\mathbb{E}(\omega^2(Z) \|U\|^2_{\mathcal{H}_1})<\infty,$
$\mathbb{E}(\|W\|^2_{\mathcal{H}_2})<\infty.$ Assume that $\sum_{k\geq 1} a_k < \infty$ and let
$$
I(h) =  \mathbb{E}\left[\langle U_{1}, U_{2}\rangle_{\mathcal{H}_1}\omega(Z_{1})\omega(Z_{2})h^{-q}K((Z_{1}-Z_{2})/h)
\exp (-\|W_1 - W_2\|^2_{\mathcal{A}}/2)\right],\qquad h>0.
$$
Then, for any $h>0$ we have
$$
\mathbb{E}(U\mid Z,W) = 0 \;\; a.s. \qquad \Leftrightarrow \qquad I(h)=0.
$$
\end{lem}

\smallskip

The reason for introducing a sequence $\{a_k\}$ with convergent partial sums is  technical. It allows for an inverse Fourier Transform formula in infinite-dimensional Hilbert spaces. In the remark following Theorem \ref{cons_Tn} we argue that considering the weighted norm
$\|\cdot \|_{\mathcal{A}}$ is not restrictive.

The idea behind the new approach we propose is to build a test statistic using an approximation of  $I(h).$ Moreover, we will let $h$ tend to zero in order to obtain an asymptotically pivotal test statistic with standard gaussian critical values. A convenient choice of the function $\omega(\cdot)$ will allow to simplify this task. As explained below, in many examples one could simply take $\omega(\cdot)\equiv 1$.

\subsection{The test statistics}\label{sec_def_test}

To estimate  $I(h)$ using the i.i.d. sample $(U_i,Z_i,W_i)$, $1\leq i\leq n$, we consider the $U-$statistic
\[
I_{n}(h) = \frac{1}{n(n-1)h^{q}}\sum\limits _{1\leq i\neq j\leq n}\left\langle U_{i}\omega(Z_{i}), \;U_{j}\omega(Z_{j}) \right\rangle_{\mathcal{H}_1} K_{ij}(h)\;
\varphi_{ij},
\]
where
\begin{equation}\label{def_phi}
K_{ij}(h) = K((Z_i - Z_j)/h), \qquad \varphi_{ij} = \exp (-\|W_i - W_j\|_{\mathcal{A}}^2/2).
\end{equation}
The variance of $I_n(h)$ could be estimated by
$$
v^2_{n}(h) = \frac{2}{n^2(n-1)^2h^{2q}}\sum\limits _{1\leq i\neq j\leq n}\left\langle U_{i}\omega(Z_{i}), \;U_{j}\omega(Z_{j}) \right\rangle^2_{\mathcal{H}_1} K^2_{ij}(h)\;
\varphi^2_{ij}.
$$
Then, the test statistic is
\begin{equation}\label{test_stat}
T_n = \frac{I_n(h)}{ v_{n}(h)}.
\end{equation}

When the $U_{i}\omega(Z_{i})$'s need  to be estimated, the test statistics becomes
\begin{equation}\label{test_stat_hat}
\widehat T_n = \frac{\widehat I_n(h)}{ \widehat v_{n}(h)},
\end{equation}
where
\[
\widehat I_{n}(h) = \frac{1}{n(n-1)h^{q}}\sum\limits _{1\leq i\neq j\leq n}\left\langle \widehat {U_{i}\omega(Z_{i})}, \;\widehat{U_{j}\omega(Z_{j})} \right\rangle_{\mathcal{H}_1} K_{ij}(h)\;
\varphi_{ij},
\]
$$
\widehat v^2_{n}(h) = \frac{2}{n^2(n-1)^2h^{2q}}\sum\limits _{1\leq i\neq j\leq n}\left\langle \widehat {U_{i}\omega(Z_{i})}, \;\widehat{U_{j}\omega(Z_{j})} \right\rangle^2_{\mathcal{H}_1} K^2_{ij}(h)\;
\varphi^2_{ij}.
$$
and the $\widehat {U_{i}\omega(Z_{i})}$ are some estimates of the
$U_{i}\omega(Z_{i})$'s.

In the example on testing the effect of a functional variable the $U_i$ are supposed observed so that $T_n$ could be used.
For the semiparametric functional partially linear models, to build $\widehat T_n$ it is convenient to take the $\omega(Z_{i})$ constant equal to 1 while the $\widehat U_{i}$ will be the residuals of the linear model with response $Y$ and covariate vector $ Z \in \mathbb{R}^q $. In the other examples of functional regression models mentioned above (functional linear regression with scalar or functional responses, quadratic functional regression, generalized functional regression, \emph{etc.}), it is convenient to set all $\omega(Z_{i})$ equal to 1 and take the $\widehat U_{i}$'s to be the residuals of the functional regression model. Below we will provide an example of argument for showing that, under suitable assumptions, replacing the $ U_{i}$'s by the $\widehat U_{i}$'s does not change the asymptotic behavior of our test statistics. Next, for variable selection in functional nonparametric regression with functional responses one can use $\widehat T_n$ and a convenient choice is  $\omega(\cdot)$ equal to the density of $Z$ and
\begin{equation*}\label{rrr}
\widehat {U_{i}\omega(Z_{i})} =  \frac{1}{n-1}\sum\limits _{k\neq i} (Y_i - Y_k)\frac{1}{g^q}L_{ik}(g),
\end{equation*}
where $L$ is another kernel, $L_{ik}(g) = L((Z_i - Z_k)/g)$ and $g$ is a bandwidth converging to zero at a suitable rate. Showing that the estimation error of the $U_i\omega(Z_i)$'s is negligible for the testing purpose requires more complicated technical assumptions but could be obtained along the lines of  the results of
\cite{Lavergne2014}. However, such en investigation is left for future work. Finally, for testing the independence between a $[0,1]-$valued random variable $Y$ and a $L^2[0,1]-$valued random variable, one could take $\omega(\cdot)\equiv 1$ and define $\widehat U_i = \mathbf{1}\{Y_i \leq t  \} - n^{-1}\sum_{j=1}^n \mathbf{1}\{Y_j\leq t  \}.$

%
%

\setcounter{equation}{0}
\section{Asymptotic theory}\label{ass_th_sec}

In this section we investigate the asymptotic properties of $T_n$ under the null hypothesis (\ref{quang}) and under a sequence of alternative hypothesis. When the $U_i\omega(X_i)$'s have to be estimated, the idea is to show that the difference $\widehat T_n - T_n$ is asymptotically negligible under suitable model assumptions. This aspect is investigated in section \ref{as_tnhat} below.

\subsection{The asymptotic critical values}\label{as_c_tn}

Under mild technical conditions we show that the test statistic is asymptotically standard normal under the null hypothesis  $\mathbb{E}[U\mid Z,W]=0$ a.s.

\setcounter{hp}{3}

\begin{hp}
\label{D}
\hspace{-0.1cm}
\begin{enumerate}
\item[(a)]
The random vectors $(U_{1},Z_{1},W_1)
,\ldots ,(U_{n},Z_{n},W_n)$ are independent draws from the random vector $(U,Z,W)\in \mathcal{H}_1 \times \mathbb{R}^q \times \mathcal{H}_2$ that satisfies  $\mathbb{E}\| U \omega(Z)\|^{4}_{\mathcal{H}_1}<\infty .$

\item[(b)]
\begin{enumerate}
\item[(i)]
The vector $Z$ admits a density $f_Z$ that is either bounded or
satisfies the condition $\int_{\mathbb{R}^q} |\mathcal{F}[f_Z]|^{2-a}(t) dt <\infty$ for some $a\in (0,1]$.
\item[(ii)] The functional covariate satisfies $\mathbb{E}[\|W\|^2_{\mathcal{H}_2}]<\infty.$

 \item[(iii)] The norm  $\|\cdot\|_{\mathcal{A}}$ is defined like in  equation (\ref{inner_a}) with a positive sequence $\{a_k\}$ such that $\sum_{k\geq 1} a_k <\infty.$
\end{enumerate}

\item[(c)]
$\exists \ \underline{\sigma}^{2},$ $C>0$ and $\nu >2$
such that:
\begin{enumerate}
\item[(i)] $0 < \underline{\sigma}^{2} \leq \mathbb{E}\left[\langle U_{1} \omega(Z_{1}),U_{2} \omega(Z_{2})\rangle^2_{\mathcal{H}_1}  \mid (Z_{1},W_1), (Z_2,W_{2})\right]$ almost surely;
\item[(ii)] $ \mathbb{E}\left[\Vert U \omega(Z)\Vert^{\nu}_{\mathcal{H}_1} \mid Z,W \right] \leq C$ almost surely.
\end{enumerate}

\end{enumerate}
\end{hp}

\setcounter{hp}{10}
\begin{hp}
\label{K} \hspace{-0.1cm}
\begin{enumerate}
\item[(a)] The kernel $K$ is multiplicative kernel in $\mathbb{R}^q$, that is $K(z_1,\cdots,z_q)= \tilde K(z_1)\cdots \tilde K(z_q)$ where $\tilde K$ is a symmetric  density of bounded variation on real line. Moreover the Fourier Transform $\mathcal{F}[\tilde K]$ is positive and integrable.
\item[(b)] $h\rightarrow 0$ and $nh^{q}/\ln n  \rightarrow
\infty.$
\end{enumerate}
\end{hp}

\begin{theor}\label{cons_Tn}
Under the Assumptions \ref{D} and \ref{K} the  test statistic $
T_{n} $ converges in law to a standard normal, provided the hypothesis (\ref{quang}) is true. Consequently, the test given by $\mathbb{I}(T_n \geq z_{1- \alpha})$, with $z_\alpha$  the $(1- \alpha)-$quantile of the standard normal distribution,  has asymptotic level  $ \alpha.$
\end{theor}

\quad

\begin{proof}[Remark 1]
Let us comment on Assumption-(\ref{D})-(ii,iii). Suppose that
the functional covariate satisfies $\mathbb{E}[\langle \langle W, \phi_k\rangle_{\mathcal{H}_2}\rangle^2_{\mathcal{H}_2}] \sim k^{-\beta},$  $\forall k\geq 1,$ for some $\beta >0.$ If $\beta > 2$ one could use directly $\|\cdot\|_{\mathcal{H}_2}$ instead of $\|\cdot\|_{\mathcal{A}}$ to build the test. Indeed, taking $a_k \sim k^{-\beta/2}$ and replacing $W$ by $\check W = \sum_{k \geq 1} b_k^{-1}\langle W, \phi_k\rangle_{\mathcal{H}_2}\phi_k,$ with $b_k=a_k^{1/2},$ one would have $\mathbb{E}(U\mid Z,W)= \mathbb{E}(U\mid Z,\check W),$  $\check W \in\mathcal{H}_2$ and $\|\check W_i - \check W_j \|_{\mathcal{A}} = \|W_i - W_j \|_{\mathcal{H}_2}.$

In the case $\beta=2,$
one could take $a_k \sim k^{-1}(\ln k)^{-(1+\varepsilon)} ,$ for some $\varepsilon >0,$
and replace $W$ by $\check W = \sum_{k \geq 1} a_k^{-1/2} (\ln k)^{-(1+\varepsilon)} \langle W, \phi_k\rangle_{\mathcal{H}_2}\phi_k.$  In this case one still has $\mathbb{E}(U\mid Z,W)= \mathbb{E}(U\mid Z,\check W)$ and $\check W \in\mathcal{H}_2.$  However,  $\|\check W_i - \check W_j \|_{\mathcal{A}}$ and $\|W_i - W_j \|_{\mathcal{H}_2}$ are no longer equal but in general remain close. Our simulation experiments reveal that in many situations where $\beta=2$ one could confidently
use  $\|W_i - W_j\|_{\mathcal{H}_2}$ instead of $\|\check W_i - \check W_j\|_{\mathcal{A}}$ to build the test.

Finally, with suitable choices, our setup covers also the range $0<\beta<2.$ When
$1<\beta<2,$ one can set $a_k \sim k^{-1}\ln^{-(1+\varepsilon)} k $  and
work with $\| W_i -  W_j \|_{\mathcal{A}}.$ For the case $0<\beta\leq 1$ one could transform $W$ in $\check W = \sum_{k \geq 1} b_k^{-1}\langle W, \phi_k\rangle_{\mathcal{H}_2}\phi_k$ with $b_k=  k^{(1-\beta)/2}\ln^{(1+\varepsilon)/2} k,$ and  take
$a_k \sim k^{\beta-2}\ln^{-2(1+\varepsilon)} k .$ The test is then built with $\| \check W_i -  \check W_j \|_{\mathcal{A}}.$

In summary,  Assumption-(\ref{D})-(ii,iii) represent mild conditions that are satisfied directly, or after simple modifications of the covariate $W,$ in most situations.
\end{proof}

\subsection{The consistency of the test}\label{alt_tn}

Let $(U^0_i,Z_i,W_i)$, $i\geq 1,$ i.i.d. such that $\mathbb{E}(U_{i}^0\mid Z_i,W_i)=0$ almost surely. Here we show that our test  is consistent against any fixed alternative and detect  Pitman alternatives
\[
H_{1n}:\ U_{in}= U^0_i + r_{n}\delta(Z_i,W_i),\quad 1\leq i\leq n, \;n\geq 1,\;
\]
with probability tending to 1, provided that the rate of decrease of the sequence  $\{r_n\}$ satisfies some conditions. These conditions are the same as for nonparametric checks of parametric regression models with finite dimension covariates.

\begin{theor}\label{consist}
Under the Assumptions \ref{D} holds for the $(U^0_i,Z_i,W_i)$'s and Assumption \ref{K} holds too.
Suppose that $\delta (\cdot,\cdot) $ and $r_n$ defining the sequence of  $H_{1n}$ satisfy the conditions
$\mathbb{E}[\delta(Z,W)] = 0,$ $0<\mathbb{E}[\Vert \delta(Z,W)\omega(Z) \Vert_{\mathcal{H}_1}^{4}]<\infty$ and  $r_{n}^{2}nh^{q/2}\rightarrow\infty.$ Then the test based on $T_n$ is consistent against the sequence of alternatives $H_{1n}$.
\end{theor}

The zero mean condition for $\delta(\cdot)$  keeps $U_{in}$ of zero mean under the alternative hypotheses $H_{1n}.$ The proof is based on  standard arguments and is relegated to the appendix.

\subsection{Goodness-of-fit test}\label{as_tnhat}

In this section we provide some guidelines on how our test could be used for testing the goodness-of-fit of functional data models. The detailed investigation of specific situations depend on the model and could not be considered in a unified framework.

In many situations, the $U_i$'s are not observed and one has to replace them by some $\widehat U_i$ obtained as residuals of some models. In this case one cannot build $T_n$ and has to work with the statistic $\widehat T_n$ defined in equation (\ref{test_stat_hat}) instead. In section \ref{emp_evid} we use some simulation experiments to show that our test could still perform well in such situations, especially with a bootstrap correction, as described in the following, when the sample size is not large enough.

From the theoretical point of view, one shall expect that the asymptotic standard normal critical values are still valid and the test is still consistent, provided that the difference $\widehat U_i - U_i$ could be controlled in a suitable way. Indeed, using the notation from section (\ref{sec_def_test}) and considering the simple case where $\omega(\cdot) \equiv 1$, we can write
\begin{eqnarray*}
\widehat I_n (h) &=& I_n (h) \\
&&+  \frac{2}{n(n-1)h^{q}}\sum\limits _{1\leq i\neq j\leq n}\left\langle U_{i}, \;\widehat U_{j} - U_{j} \right\rangle_{\mathcal{H}_1} K_{ij}(h)
\varphi_{ij}\\
&&+ \frac{1}{n(n-1)h^{q}}\sum\limits _{1\leq i\neq j\leq n}\left\langle \widehat  U_{i} - U_{i}, \;\widehat U_{j} - U_{j} \right\rangle_{\mathcal{H}_1} K_{ij}(h)
\varphi_{ij}\\
&\stackrel{def}{=}& I_n (h) + R_{1n} (h)+ R_{2n} (h) .
\end{eqnarray*}
Next, one has to control $\widehat U_i - U_i$ and hence $R_{1n} (h)$ and $R_{2n} (h)$ and this strongly depends on the specific model considered. Many functional data models would allow to show that $R_{1n} (h)$ and $R_{2n} (h)$ are negligible under reasonable conditions in the model (regularity conditions on the model parameter and  the functional covariate $W$) and for suitable rates of the bandwidth.
For instance, \cite{Patilea2012a} investigated in detail the case of linear model with scalar responses. Their investigation could be adapted to our test and obtain similar conclusions.
In the case of a functional linear model with $L^2 [0,1]$ responses and finite and infinite dimension covariates one would observe a sample of $(Y,Z,W)$ where
$$
Y(t) = Z^\top \beta + \langle\xi(t,\cdot), W\rangle_{\mathcal{H}_2} + U (t), \qquad t\in[0,1].
$$
Since $\beta$ is expected to be estimated at parametric rate, the control of $\widehat U_i - U_i$ would depend on the conditions on the rate of convergence of $\widehat \xi(\cdot,\cdot),$ the estimate of $\xi(\cdot,\cdot).$ Under suitable but mild conditions, one could expect the rate of $R_{1n} (h)$ to be of order $n^{-1}h^{-q/2}$ times the norm of
$\widehat \xi(\cdot,\cdot) - \xi(\cdot,\cdot),$ while the rate of $R_{2n} (h)$ to given by the square of
the norm of
$\widehat \xi(\cdot,\cdot) - \xi(\cdot,\cdot).$ Meanwhile, the rate of $I_n(h)$ is $O_{\mathbb{P}}(n^{-1}h^{-p/2}).$ The required restrictions on the bandwidth to preserve the asymptotic standard normal critical values follow. Let us point out that
slower rates for
the norm of
$\widehat \xi(\cdot,\cdot) - \xi(\cdot,\cdot)$
will require faster decreases for $h,$
and this will result in a loss of power against sequences of local alternatives.

\subsection{Bootstrap  critical values}

To correct the finite sample critical values let us propose a simple wild bootstrap procedure. The bootstrap sample, denoted by $U_i^\star$ , $1\leq i\leq n$, is defined as  $U_i^\star = \zeta_i U_i$, $1\leq i\leq n$, where $\zeta_i$, $1\leq i\leq n$ are independent random variables following the two-points distribution proposed by \cite{Mammen1993}. That means $\zeta_i=-(\sqrt{5}-1)/2$ with probability $(\sqrt{5}+1)/(2\sqrt{5})$ and $\zeta_i=(\sqrt{5}+1)/2$ with probability $(\sqrt{5}-1)/(2\sqrt{5})$. A bootstrap test statistic $T_n^\star$ is built from a bootstrap sample as was the original test statistic.
When this scheme is repeated many times, the bootstrap critical value $z^\star_{1-\alpha, n}$ at level $\alpha$ is the empirical $(1-\alpha)-$th quantile of the bootstrapped test statistics. The asymptotic validity of this bootstrap procedure is guaranteed by the following result. It states that the bootstrap critical values are asymptotically standard normal under the null hypothesis and under the alternatives like in section \ref{alt_tn}. The proof could be obtained by rather standard modifications of the proof of Theorem \ref{cons_Tn} and hence will be omitted.

\begin{theor}\label{bbot}
Suppose that the conditions of Theorem \ref{consist} hold true, in particular in the case $r_n\equiv 0.$ Then
$$
\sup_{x\in\mathbb{R}} \left| \mathbb{P}\left( T_n^\star \leq x \mid U_1,Z_1,W_1,\cdots, U_n,Z_n,W_n \right) - \mathbb{P}( T_n \leq x ) \right| \rightarrow 0,\qquad \text{in probability.}
$$
\end{theor}

\setcounter{equation}{0}
\section{The error in covariates case}
\label{est_basis}

In this section we show that our testing procedure extends to the case where the covariates are observed with error. In some applications, the observations $Z_i$ and $W_i$ are not directly observed but could be  estimated by some $\widehat Z_i$ and $\widehat W_i$ computed from the data.
To better illustrate the methodology, let us focus on the test for the effect of a functional variable. For this reason, in this section let us take $q=1,$ $Z = \langle \widetilde W, \phi_1  \rangle_{\mathcal{H}_2}$ and $W = \widetilde W - \langle \widetilde W, \phi_1 \rangle_{\mathcal{H}_2} \phi_1,$ where $\widetilde W \in \mathcal{H}_2$
and $\phi_1, \phi_2, \cdots $ are the elements of an orthonormal basis in  $\mathcal{H}_2.$

In functional data analysis where usually $\mathcal{H}_2 = L^2[0,1]$ the choice of the basis is a key point.
The statistician would likely prefer a basis allowing an accurate representation of $\widetilde W$ with a minimal number of basis elements.
A commonly used basis is given by the eigenfunctions of the covariance operator $\mathcal{K}$ that is defined by
$
(\mathcal{K} v)(\cdot) = \int \mathcal{K} (\cdot,s) v(s) ds,$ $v\in L^2[0,1],
$
where $\widetilde W $ is supposed to satisfy  $\int \mathbb{E}(\widetilde W ^2(t)) dt < \infty$ and $$\mathcal{K}(t,s) = \mathbb{E}[ \{\widetilde W (t) - \mathbb{E}(\widetilde W (t))  \}  \{ \widetilde W (s) - \mathbb{E}(\widetilde W (s)) \}] $$ is supposed positive definite. Let $\theta_1, \theta_2 , \cdots $ denote the  eigenvalues of $\mathcal{K}$ and let $\mathcal{R}=\{\phi_1,\phi_2,\cdots\}$ be the corresponding basis of eigenfunctions  that are usually  called the functional principal components (FPC). The FPC  orthonormal basis  provide optimal (with respect to the mean-squared error) low-dimension representations of $\widetilde W .$ See, for instance, \cite{Ramsay2005}.
In most of the applications the FPC are unknown and has to be estimated from
\begin{equation*}
(\widehat {\mathcal{K}} v)(t) = \int_{[0,1]} \widehat {\mathcal{K}}(t,s) v(s) ds ,\qquad t\in [0,1],
\end{equation*}
where
\begin{equation}\label{K_hat}
\widehat {\mathcal{K}}(t,s) = n^{-1} \sum_{i=1}^n  \left\{\widetilde W _i(t) - n^{-1} \sum_{j=1}^n  \widetilde W _j(t)  \right\}  \left\{ \widetilde W _i(s) - n^{-1} \sum_{j=1}^n  \widetilde W _j(s) \right\} .
\end{equation}
Let $\widehat \theta_1, \widehat \theta_2 , \cdots\geq 0$ denote the eigenvalues of $\widehat {\mathcal{K}}$ and let $\widehat \phi_1,\widehat \phi_2,\cdots$ be the corresponding basis of eigenfunctions, that is the estimated FPC. For identification purposes, we adopt the usual  condition  $\langle  \phi_j, \widehat \phi_j\rangle \geq 0,$ $\forall j$. Now, we can define  $\widehat Z_i = \langle \widetilde W_i, \widehat \phi_1  \rangle_{\mathcal{H}_2} $ and $\widehat W_i = \widetilde W_i - \langle \widetilde W_i, \widehat \phi_1 \rangle_{\mathcal{H}_2} \widehat \phi_1,$ the estimates of $Z_i$ and $W_i.$

Having in mind such types of situations, herein we will suppose that
\begin{equation}\label{approx_cond}
\|\widehat Z_i - Z_i \| + \|\widehat W_i - W_i \|_{\mathcal{H}_2} \leq \Gamma_i \Delta_n,\qquad 1\leq i\leq n,
\end{equation}
where $\Gamma_i $ are independent copies of some random variable $\Gamma$ that depend on $X,$ and $\Delta_n$ depend on the data but could be taken the same for all $i.$  For $\Delta_n$ and $\Gamma$ we will suppose
\begin{equation}\label{exp_cond}
\Delta_n = O_{\mathbb{P}}(n^{ -1/2})\qquad \text{ and } \qquad\exists a>0 \quad \text{ such that } \quad \mathbb{E}\exp(a \Gamma) <\infty.
\end{equation}
Clearly, alternative conditions on the rate of $\Delta_n$ and the moments of $\Gamma$ could be considered, resulting in alternative conditions on the bandwidths in the statements below. As it will be explained below, the conditions (\ref{exp_cond}) are convenient for the example of $\widehat Z_i$ and $\widehat W_i$ obtained from estimated FPC basis.

Let us introduce some notation
\begin{equation}\label{def_phi_hat}
\widehat K_{ij}(h) = K((\widehat Z_i - \widehat Z_j)/h), \qquad \widehat \varphi_{ij} = \exp (-\|\widehat W_i - \widehat W_j\|_{\mathcal{A}}^2/2).
\end{equation}

Let
\begin{equation*}
\widetilde T_n = \frac{\widetilde I_n(h)}{\widetilde v_{n}(h)}.
\end{equation*}
where $\widetilde I_n(h)$ and $\widetilde v_{n}(h)$ are defined as $I_n(h)$ and $v_{n}(h)$ in section \ref{sec_def_test} but with $\widehat Z_i$ and $\widehat W_i$ instead of $Z_i$ and $W_i.$

\begin{theor}\label{test_estim}
Suppose that $q=1,$ the Assumptions \ref{D}-(a),  \ref{D}-(b)-(ii, iii), \ref{K}-(a)  are met and conditions (\ref{approx_cond}) and (\ref{exp_cond}) hold true. Assume one of the following conditions is met:
\begin{enumerate}
\item   $nh^{4}/\ln^{2} n \rightarrow \infty$ and  $\int_{\mathbb{R}^q} |\mathcal{F}[f_Z]|^{2-a}(t) dt <\infty$ for some $a\in (0,1];$
\item $nh^{2}/\ln^2 n \rightarrow \infty$ and $f_Z$ is bounded;
\item $nh/\ln n \rightarrow \infty$  and $f_Z$ is uniformly continuous.
\end{enumerate}
Then the Theorems \ref{cons_Tn},  \ref{consist} and \ref{bbot} remain valid with the test statistic $T_n$ replaced by $\widetilde T_n.$
\end{theor}
\medskip

The proof of Theorem \ref{test_estim} is a direct consequence of  Lemma \ref{diff_l2b} in the Appendix and is hence omitted.

Let us revisit the problem of the test for the effect of a functional variable,
where $Z = \langle \widetilde W, \phi_1  \rangle_{\mathcal{H}_2}$ and $W = \widetilde W - \langle \widetilde W, \phi_1 \rangle_{\mathcal{H}_2} \phi_1.$ The conditions on the random variable $Z$ required in Lemma \ref{diff_l2b} are mild conditions satisfied in the common examples of functional covariates considered in the literature. Concerning condition (\ref{approx_cond}), consider the operator norm $\| \mathcal{K} \|_{S}$ defined by $$\| \mathcal{K} \|^2_{S}= \int \int \sigma^2(t,s)dt ds.$$ Under Assumptions \ref{D}-(a),  \ref{D}-(b)-(ii), the empirical covariance operator satisfies
$$
\|\widehat {\mathcal{K}} - \mathcal{K} \|_{S} = O_{\mathbb{P}}(1/\sqrt{n}),
$$
see for instance \cite{Bosq2000} or  \cite{Horvath2012}.
On the other hand, suppose that $\theta_1,$ the eigenvalue associated to $\phi_1,$ is different from all the others eigenvalues of the operator $\mathcal{K}.$ By Lemma 4.3 in \cite{Bosq2000} or Lemma 2.3 in \cite{Horvath2012}, and the fact that the spectral norm of the operator $\widehat {\mathcal{K} } - \mathcal{K} $ is smaller or equal to  $\|\widehat {\mathcal{K}} - \mathcal{K} \|_{S}$,
$$
  \| \widehat\phi_1 - \phi_1 \|^2 \leq \frac{8}{\varsigma^2} \|\widehat {\mathcal{K}} - \mathcal{K} \|_{S}^2,
$$
where $\varsigma $ is the distance between $\theta_1$ and the set $\{\theta_2,\theta_3,\cdots\}$ of all the other eigenvalues of $\mathcal{K}.$ Here the eigenvalues of $\mathcal{K}$ are not necessarily ordered, $\theta_1$ could be any eigenvalue separated from all the others. Deduce that condition (\ref{approx_cond}) is guaranteed for instance if there exists $a>0$ such that
$$
\mathbb{E}\exp(a\| \widetilde W\|_{\mathcal{H}_2}) <\infty.
$$
The exponential moment condition is met if, for instance,  $\widetilde W$ is a mean-zero Gaussian process defined on the unit interval with  $\sup_{t\in[0,1]} \mathbb{E} [\widetilde W ^2(t)]<\infty ;$ see chapter A.2 in \cite{Vaart1996}. Moreover, in general, a moment restriction on $\widetilde W$ is not restrictive for significance testing. Indeed, if $\widetilde W$ does not satisfy such a condition, it suffices to transform  $\widetilde W$ into some variable  $\in \mathcal{H}_2$ such that  $\widetilde V$ generates the same $\sigma-$field  and  $\widetilde V$ satisfies the required moment condition.

\setcounter{equation}{0}
\section{Empirical evidence}\label{emp_evid}

In this section we illustrate the empirical performances of our  testing procedure. For that purpose,
we consider both scalar and functional responses $Y$. We used an Epanechnikov kernel in our applications, that is $K\left(x\right)=0.75\left(1-x^{2}\right)\mathbf{1}\left\{ \left|x\right|<1\right\}. $ We calculated  $\varphi_{ij}$ and $\widehat \varphi_{ij}$
in two ways: with the norm in the Hilbert space $L^2[0,1]$ of the  covariate and with the norm $\|\cdot\|_{\mathcal{H}}$ proposed in Remark 1 for the case $\beta=2.$

Below $\left\langle \cdot,\cdot\right\rangle $
is the usual inner product on $L^{2} \left[0,\,1\right], $ that is
$
\left\langle f,\, g\right\rangle =\intop_{0}^{1}f\left(t\right)g\left(t\right)dt.
$
Let   $\hat{\mathcal{K}}$ be the
empirical covariance operator defined in (\ref{K_hat}) and let  $\hat{\theta}_{1}\geq\hat{\theta}_{2}\geq\dots\geq0$ be its eigenvalues
and $\widehat\phi_{1}$, $\widehat\phi_{2}$, $\dots$ the corresponding eigenfunctions.

\subsection{The scalar response case}

We simulate data samples of size $n=40$ using the models
\begin{eqnarray}
Y_{i} & = & a+\left\langle X_{i},\, b\right\rangle +\delta\left\langle X_{i},\, b\right\rangle ^{2}+ U_{i},\label{eq:scalY}\\
Y_{i} & = & \left(\lambda_k^{-1}\left\langle X_i , \,e_k\right\rangle^2 - 1\right) + U_{i},\qquad\qquad 1\leq i\leq n,\label{eq:scalYfar}
\end{eqnarray}
where $X_{i}$ is a Wiener process, $U_{i}$ are independent centered normal variables with variance
$\sigma^2= 1/16$,
$$
a=0 \qquad \text{and} \qquad b\left(t\right)=\sin^{3}\left(2\pi t^{3}\right), \quad t\in\left[0,\,1\right].
$$
Moreover,
$$
e_k\left(t\right)=\sqrt{2}\sin\left(\left(k-1/2\right)\pi t\right) ,\quad
t\in[0,1],
$$
and $\lambda_k=(k-1/2)^{-2}\pi^{-2}$ and $k$ is some fixed positive integer.
The null hypothesis corresponds to $\delta=0$ while nonnegative $\delta$'s yield quadratic alternatives.

We then estimate $b$ using the functional principal component approach, see see, e.g., \cite{Ramsay2005} and \cite{Horvath2012}. The first five  principal
components of the $X_{i}$s are used so that $b$ is estimated by
\[
\widehat{b}\left(t\right)=\sum_{j=1}^{5}\widehat{b}_{j}\widehat \phi_{j}\left(t\right),
\]
where $\widehat{b}_{j}=\widehat{\theta}_{j}^{-1}\widehat{g}_{j}$, $\widehat{g}_{j}=\left\langle \widehat{g,}\,\widehat{\phi}_{j}\right\rangle $
with
\[
\widehat{g}\left(t\right)=\dfrac{1}{n}\sum_{i=1}^{n}\left(Y_{i}-\overline{Y}_{n}\right)\left(X_{i}\left(t\right)-\overline{X}_{n}\left(t\right)\right)
\]
and $a$ by $\widehat{a}=\overline{Y}_{n}-\left\langle \overline{X}_{n},\,\widehat{b}\right\rangle $.
The test statistics are built with $q=1,$
$$
\widehat{U}_{i}=Y_{i}-\widehat{a}-\left\langle X_{i},\,\widehat{b}\right\rangle, \quad \widehat{Z}_{i}=\dfrac{\widetilde{Z}_{i}}{\sqrt{n^{-1}\sum_{j=1}^{n}\widetilde{Z}_{j}^{2}}} \quad\mbox{ and } \quad\widehat{W}_{i}=\dfrac{\widetilde{W}_{i}}{\sqrt{n^{-1}\sum_{j=1}^{n}\Vert \widetilde{W}_{j}\Vert_{\mathcal{H}_2} ^{2}}}\;,
$$
where
$
\tilde Z_i = \langle X_i, \widehat \phi_{1}\rangle$ and $\tilde W_i = X_i -  \langle X_i, \widehat \phi_{1}\rangle \widehat \phi_{1} .
$

First, we investigate the accuracy of the asymptotic critical values and the effectiveness
of the  bootstrap correction, with $199$ bootstrap samples, for level $\alpha=10\%$.
Several bandwidths are considered, that is $h=cn^{-1/5}$ with $c\in\{2^{k/2},\,k=-2,-1,0,1,2\}$.
The results of 5000 replications are plotted in the left panel of Figure \ref{fig:Level}.
The  normal critical values are quite inaccurate, while the bootstrap corrections are very effective,
whatever the considered bandwidth is. The differences between the results for the statistics defined with $\|\cdot\|_{\mathcal{H}_2}$ and those for the statistics defined with $\|\cdot\|_{\mathcal{A}}$ are imperceptible.

Next, we compare our test to the one introduced by  \cite{Patilea2012a} (hereafter PSSa) based on projections.  The test statistic of PSSa is
\[
T_{n}^{PSSa} = \dfrac{Q_{n}(\widehat{\gamma}_{n};\,\widehat{a},\,
\widehat{b})}{\widehat{v}_{n}(\widehat{\gamma}_{n};\,\widehat{a},\,\widehat{b})}
\]
where
\[
Q_{n}(\gamma;\,\widehat{a},\,\widehat{b})=\dfrac{1}{n\left(n-1\right)}\sum_{1\leq i\neq j\leq n}\widehat{U}_{i}\widehat{U}_{j}\dfrac{1}{h}K\left(h^{-1}\left\langle X_{i}-X_{j},\,\gamma\right\rangle \right),\quad \gamma\in\mathbb{R}^p,
\]
and $\widehat{v}_{n}^{2}(\gamma;\,\widehat{a},\,\widehat{b})$
is an estimation of the variance of $nh^{1/2}Q_{n}(\gamma;\,\widehat{a},\,\widehat{b}).$
Here and in the following, the vector $ \gamma = (\gamma_1,\cdots,\gamma_p)^\top$  is identified with $\sum_{k=1}^p \gamma_k \phi_k\in L^2[0,1].$
The value of $p$ is chosen by the statistician. The direction $\widehat \gamma_n$ is selected as
\[
\widehat{\gamma}_{n}=\arg\max_{\gamma\in B_{p}}\left[nh^{1/2}Q_{n}(\gamma,\widehat a,\widehat b)/\widehat{v}_{n}
(\gamma,\widehat a,\widehat b)-\alpha_{n}\mathbf{1}\left\{ \gamma\neq\gamma_{0}\right\} \right],
\]
where $B_{p}\subset\mathcal{S}^{p}=\left\{ \gamma\in\mathbb{R}^{p}:\left\Vert \gamma\right\Vert =1\right\} $
is a set of positive Lebesgue measure on $\mathcal{S}^{p}$ and $\gamma_0$ is a privileged direction chosen by the statistician and $\alpha_n$ is a penalty term.
Here we follow PSSa and we take $p=3$ and $B_{3}$ as a set of $1200$ points on $\mathcal{S}^3$, $\gamma_0=\left(1,1,1\right)/\sqrt{3}$ and  $\alpha_{n}=3$.

The results are presented on Figure
\ref{fig:Scalar-Y}
the null hypothesis (5000 replications) and several alternatives (2500 replications) defined by some positive values of $\delta.$
The PSSa statistic is computed with wild bootstrap critical
values. The rejection rate for the bootstrap version of our test appears to be better than that based on asymptotic critical values
for each considered alternative.
Moreover, the results obtained with $\|\cdot\|_{\mathcal{H}_2}$ are better than  those obtained  with $\|\cdot\|_{\mathcal{A}}.$
The  PSS1 outperforms our test in terms of power for the setups (\ref{eq:scalY}) and (\ref{eq:scalYfar})
with $k=2$. This could be explained by the nature of the PSS1 statistic which by construction is powerful against such alternatives. When considering the setup
(\ref{eq:scalYfar})
with $k=4$ the power is deteriorates drastically for all the tests. The fourth coordinate $\langle X_i,e_4\rangle$ being independent of the first three involved in the PSS1 statistic, the empirical power of that statistic is practically equal to the level for any sample size. The empirical power of our statistic improves with the sample size and so confirms the asymptotic results.
The plateau for the empirical rejection curves for our test could be explained by an inflated variance on the alternatives, but its level increases with the sample size.

\subsection{The case of functional response}

Three models with functional $Y$ are considered:
\begin{eqnarray}
Y_{i}\left(t\right) & = & \delta\times\beta\left(t\right)X_i\left(t\right)+\epsilon_{i}\left(t\right)\label{eq:funcYconc}\\
Y_{i}\left(t\right) & = & \delta\times H\left(B_{i}\left(t\right)\right)+\epsilon_{i}\left(t\right)\label{eq:funcYquad}\\
Y_{i}\left(t\right) & = & \delta\times\lambda^{-1/2}_k\left\langle B_i , e_k \right\rangle+\epsilon_{i}\left(t\right)\label{eq:funcYfar}
\end{eqnarray}
$1\leq i\leq n$, where $X_{i}$ and $\epsilon_{i}$
are independent Brownian bridges, $B_{i}$ is a Brownian motion,
$$
\beta\left(t\right)=\exp\left\{ -4 \left(t-0.3\right)^{2}\right\} ,\quad
t\in[0,1],
$$
$e_k(\cdot)$ and $\lambda_k$ are defined as in the case of scalar response for some fixed  $k\geq 1$,
and
$H\left(x\right)=x^{2}-1$, $x\in\mathbb{R}$. We consider $q=1$ and the $\widehat Z_i$, $\widehat W_i$ and $\widehat \varphi_{ij}$ are built like in the case of a scalar response.

We compare our test with the one considered by \cite{Patilea2012b} (hereafter PSSb). Their statistic, let us call it $T_n^{PSSb},$ which is a variant of $T_n^{PSSa}$ above defined  with a different $Q_{n}. $ That is, in the definition of
$Q_{n}$
the product $\widehat{U}_{i}\widehat{U}_{j}$ is replaced by the scalar product
$\langle \widehat{U}_{i},\,\widehat{U}_{j}\rangle $ and  $K\left(h^{-1}\left\langle X_{i}-X_{j},\,\gamma\right\rangle  \right)$
by
\[
K_{h}\left(h^{-1}[F_{\gamma,n}\left(\left\langle X_{i},\,\gamma\right\rangle  \right)-F_{\gamma,n}\left(\left\langle X_{j},\,\gamma\right\rangle \right)]\right)
\]
where $F_{\gamma,n}$ is the empirical c.d.f. of the sample
$\left\langle X_{1},\,\gamma\right\rangle ,\dots,\,\left\langle X_{n},\,\gamma\right\rangle ,$
$\gamma\in B_p\subset\mathbb{R}^p .$ Following PSS2, in this case we take $p=3,$ $B_3$
as a set of $1200$ points on $\mathcal{S}^3$, $\gamma_0=\left(1,1,1\right)/\sqrt{3}$
and $\alpha_{n}=2$. Moreover, since here we test for the effect, $\widehat U_i$ are nothing but the observations $Y_i.$

We also compare our test with the test  of \cite{Kokoszka2008} (hereafter KMSZ) based on the eigenvalues $\left(\widehat{\gamma}_{k}\right)_{k}$
and $(\widehat{\lambda}_{k})_{k}$ and eigenvectors $\left(\widehat{u}_{k}\right)_{k}$
and $\left(\widehat{v}_{k}\right)_{k}$, $1\leq k\leq n$ of the respective
empirical operators
$$
\Gamma_{n}x = \dfrac{1}{n}\sum_{i=1}^{n}\left\langle X_{i},\, x\right\rangle X_{i},\qquad
\Lambda_{n}x = \dfrac{1}{n}\sum_{i=1}^{n}\left\langle Y_{i},\, x\right\rangle Y_{i},
$$
and also
\[
\Delta_{n}x=\dfrac{1}{n}\sum_{i=1}^{n}\left\langle X_{i},\, x\right\rangle Y_{i},
\]
the test statistic being
\[
T_{n}\left(\tilde p,\tilde q\right)=n\sum_{k=1}^{\tilde p}\sum_{j=1}^{\tilde q}\dfrac{\left\langle \Delta_{n}\widehat{v}_{k},\,\widehat{u}_{j}\right\rangle ^{2}}{\widehat{\gamma}_{k}\widehat{\lambda}_{j}}.
\]
This statistic is asymptotically $\chi^{2}\left(\tilde p \tilde q\right)$ distributed when there is no
linear effect of $X$ on $Y$. We test the ``no effect'' model on
the three setups (\ref{eq:funcYconc}), (\ref{eq:funcYquad}) and (\ref{eq:funcYfar}) using
$\widehat{U}_{i}=Y_{i}-\overline{Y}_{n}$. For this we consider the cases $\tilde p=1,$ $\tilde q=6$ and $\tilde p=2,$ $\tilde q=6.$

Again, we investigate the accuracy of the asymptotic critical values and the bootstrap correction,
following the same steps as in the case of scalar $Y, $ this time for $1000$ replications under the null hypothesis. We present
the results in the right panel of Figure \ref{fig:Level}. The conclusions are similar to those of the scalar case, that is the asymptotic critical values are rather
inaccurate with $n=40$. The bootstrap correction is  quite effective, whatever the considered
bandwidth is.
The empirical power results for positive deviations $\delta$ for the three models considered
are presented in Figure \ref{fig:Functional-Y}. They are based on a number of  500 replications of the experiment.
The results obtained with $\|\cdot\|_{\mathcal{H}_2}$ are again preferable.
One can see that KMSZ and PSS2 perform very well for the concurrent alternative.
However, for a linear alternative with $k=4$, the bootstrap version of our test seems to be the best choice. The good performance of the PSS2 with samples of size $n=40$ could be explained by a correlation between $\langle B_i, \widehat \phi_1 \rangle, \cdots,\langle B_i, \widehat \phi_3 \rangle$ and $\langle B_i, \widehat \phi_4 \rangle$ which approximates $\langle B_i,  e_4 \rangle.$ This correlation vanishes when $n$ increases resulting in a loss of power for PSS2 test. In this experiment we also studied the effect of larger dimension $q$ with $n=40$ and the concurrent alternative, equation (\ref{eq:funcYconc}), and quadratic alternative, equation (\ref{eq:funcYquad}). The results presented in Figure \ref{fig:Functional-Y} reveals a drastic decrease of power. A possible explanation is that when the first components $\langle X_i,\widehat \phi_1\rangle$ carry enough information on the covariate, the price to pay in terms of power for smoothing in higher dimension could be too high, so that it may be preferable to consider  $q=1.$

\subsection{Real data application}

The approach proposed in this paper is applied to check the goodness-of-fit of several models for the Canadian weather dataset. This dataset is studied in \cite{Ramsay2005} and is included in the R package fda (http://www.r-project.org). The data consist of the daily mean temperature  and rain   registered in 35 weather stations in Canada. A curve is available for each station, describing the rainfall for each day of the year. This is the functional response.
The same type of curve with the temperature is used as functional predictor. Several regression models with functional covariate and functional response have been studied in \cite{Ramsay2005}, and illustrated with the Canadian weather dataset. The purpose here is to  assess the validity of each of the following three models
\begin{eqnarray}
Y_{ij}\left(t\right) & = & \mu\left(t\right)+\varepsilon_{ij}\left(t\right),\label{eq:NoEffect}\\
Y_{ij}\left(t\right) & = & \mu\left(t\right)+\alpha_{j}\left(t\right)+\varepsilon_{ij}\left(t\right),\label{eq:Anova}\\
Y_{ij}\left(t\right) & = & \mu\left(t\right)+\alpha_{j}\left(t\right)+\intop X_{ij}\left(s\right)\xi\left(s,t\right)ds+\varepsilon_{ij}\left(t\right),\label{eq:Ancova}
\end{eqnarray}
where $\sum_{j=1}^{J}\alpha_{j}\left(\cdot\right)\equiv0$
to ensure identification of models (\ref{eq:Anova}) and (\ref{eq:Ancova}).
The stations are classified in four climatic zones (Atlantic, Pacific, Continental, Arctic) and  $Y_{ij}(t)$ represents the logarithm of the rainfall at the station $i$ of the climate zone $j$ on day $t$, $X_{ij}(t)$ is the temperature at the same station on day $t$ of the year.
Since each observation
$Y_{ij}$ is observed for the same time design, we just use
\[
\widehat{U}_{ij}\left(\cdot\right)=Y_{ij}\left(\cdot\right)-\overline{Y}_{n}^{\left(i\right)}\left(\cdot\right)\mbox{ and }\widehat{U}_{ij}\left(\cdot\right)=Y_{ij}\left(\cdot\right)-\overline{Y}_{\cdot j}^{\left(i\right)}\left(\cdot\right)
\]
for models (\ref{eq:NoEffect}) and (\ref{eq:Anova}) respectively. Here we use the notation $$\overline{A}_{n}^{\left(i\right)}=\left(n-1\right)^{-1}\left(-A_{ij}+\sum_{j=1}^{J}\sum_{k=1}^{n_{j}}A_{kj}\right)$$
and $\overline{A}_{\cdot j}^{\left(i\right)}=\left(n_{j}-1\right)^{-1}\sum_{k\in\{1,\dots,n_{j}\}\backslash\{i\}}A_{kj}$
represent respectively the leave-one out overall mean and the class $j$
mean for the variable $A$ and the observation $i$.
For the model (\ref{eq:Ancova}), let us notice that
\[
\overline{Y}_{\cdot j}^{\left(i\right)}\left(t\right)=\mu\left(t\right)+\alpha_{j}\left(t\right)+\intop\overline{X}_{\cdot j}^{\left(i\right)}\left(s\right)\xi\left(s,t\right)ds+\overline{\varepsilon}_{\cdot j}^{\left(i\right)}\left(t\right)
\]
and then
\[
\tilde{Y}_{ij}\left(t\right)=\intop\tilde{X}_{ij}\xi\left(s,t\right)ds+\tilde{\varepsilon}_{ij}
\]
where $\tilde{A}_{ij}=A_{ij}-A_{\cdot j}^{\left(i\right)}$. Thus we
construct the functional principal components based on $\left\{ \tilde{X}_{ij}\left(\cdot\right),\, j\in\left\{ 1,\dots,J\right\} ,\, i\in\left\{ 1,\dots,n_{j}\right\} \right\} $
which leads to $\tilde{X}_{ij}\left(\cdot\right)\simeq\sum_{\kappa=1}^{K}\lambda_{\kappa}c_{ij\kappa}v_{\kappa}\left(\cdot\right)$
(where $\sum_{j=1}^{J}\sum_{i=1}^{n_{j}}c_{ij\kappa}^{2}=1$, $\Vert v_{\kappa}\left(\cdot\right)\Vert^{2}=1$
and $\lambda_{1}\geq\lambda_{2}\geq\dots\geq0$) and
\[
\widehat{U}_{ij}\left(t\right)=\tilde{Y}_{ij}\left(t\right)-\sum_{\kappa=1}^{K}c_{ij\kappa}\left(-c_{ij\kappa}\tilde{Y}_{ij}\left(t\right)+\sum_{l=1}^{J}\sum_{k=1}^{n_{l}}c_{kl\kappa}\tilde{Y}_{kl}\left(t\right)\right).
\]
All this leave-one-out feature is used to avoid overfitting and for
the choice of $K=13$, we used the one that minimizes $\sum_{j=1}^{J}\sum_{i=1}^{n_{j}}\Vert\widehat{U}_{ij}\left(\cdot\right)\Vert^{2}_{\mathcal{H}_2}$.
We also consider the effect of this choice considering also $K=12$
and $K=14$.

On one hand, we choose not to project the response variable before
the test process, because some of the link between $Y$ and $X$ could
be in the truncated part of $Y$. On the other hand, reducing the dimension
for $X$ is compulsory to solve the infinite dimension inverse problem.
We consider the smoothing dimensions $q=1$ and $q=3$, with $h=n^{-1/\left(q+4\right)}$
for the test. Only the norm $\|\cdot\|_{\mathcal{H}_2}$ was used for the functional covariates. Our test rejects all the models when using $q=1$. Meanwhile the model (\ref{eq:Ancova}) is not rejected with $q=3.$ This could be explained by a possible lack of power due to smoothing in higher dimension.

%

\begin{table}
\begin{centering}
\begin{tabular}{|c|c|c|}
\hline
Model & $q=1$ & $q=3$\tabularnewline
\hline
\hline
(\ref{eq:NoEffect}) & $0$ & $0$\tabularnewline
\hline
(\ref{eq:Anova}) & $0$ & $0.009$\tabularnewline
\hline
(\ref{eq:Ancova}), $K=12$ & $0.016$ & $0.403$\tabularnewline
\hline
(\ref{eq:Ancova}), $K=13$ & $0.023$ & $0.736$\tabularnewline
\hline
(\ref{eq:Ancova}), $K=14$ & $0.015$ & $0.723$\tabularnewline
\hline
\end{tabular}

\caption{Bootstrap $p-$values for modeling Canadian Weather data three different ways and
for different smoothing dimension $q\in\left\{ 1,\,2\right\} $ and
$999$ bootstrap samples.\label{tab:p-values-canadian}}
\end{centering}
\end{table}


%
\newpage

\small
\bibliographystyle{ecta}
\bibliography{ManuscritAbre}
\normalsize

\newpage

\setcounter{equation}{0}
\section{Technical results and proofs}

\begin{proof}[Proof of Lemma \ref{lemma_fund}]
The implication from left to right is obvious. For the reverse one,
let us consider $l^2\subset\mathbb{R}^\infty$ the space of real valued, square integrable sequences $x=(x_1,x_2,\cdots),$ endowed with the scalar product $\langle x , y \rangle_2 = \sum_{k=1}^\infty x_ky_k.$ Since any $w\in\mathcal{H}_2$ can be decomposed $ w = \sum_{k\geq 1}\langle w,\phi_k\rangle_{\mathcal{H}_2} \phi_k ,$ where $\{  \phi_1,\phi_2,\cdots \}$ is the orthonormal basis considered in $\mathcal{H}_2,$ we shall use the usual identification between $\mathcal{H}_2$ and $l^2$ given by the isomorphism
$w\in\mathcal{H}_2 \mapsto (\langle w,\phi_1\rangle_{\mathcal{H}_2},\langle w,\phi_2\rangle_{\mathcal{H}_2},\cdots) \in l^2.$
Denote $W_{12} = W_1 - W_2.$

Next, consider the linear operator $Q$ from $\mathcal{H}_2$ into $\mathcal{H}_2$ defined by
$$
Q \phi_k = a_k \phi_k , \qquad k\geq 1.
$$
The condition that the series  $\sum_{k\geq 1} a_k$ is convergent means that the trace of the operator $Q$ is finite. Now, since $\mathbb{E}[ \| W_{12} \|^2_{\mathcal{H}_2} ] < \infty,$ there exists a set of events $N$ such that $\mathbb{P}(N)=1$ and for any $\omega\in N,$ $W_{12}(\varpi)\in l^2$ and hence $Q(W_{12}(\varpi))\in l^2.$
By classical results in mathematical analysis in infinite-dimensional Hilbert spaces, see for instance Theorem 1.12 in \cite{DaPrato2006}, there exists a (unique) probability measure $\mu_Q$ on $\mathcal{H}_2$ endowed with the Borel $\sigma-$field such that for any $\varpi\in N,$
\begin{eqnarray*}
\exp(-\| W_{12} (\varpi)\|^2_{\mathcal{A}}/2) &=& \exp(-\langle  Q(W_{12} (\varpi)), \; W_{12} (\varpi)\rangle_ {\mathcal{H}_2}/2) \\
&=& \int_{\mathbb{R}^\infty} \exp\left( i \left \langle W_{12} (\varpi),x\right \rangle_2 \right)\mu_Q (x)\\
&=& \int_{l^2} \exp\left( i \left \langle W_{12}  (\varpi),x\right \rangle _2 \right)\mu_Q (x).
\end{eqnarray*}
The last equality expresses  the fact that the probability measure $\mu_Q $ concentrates on $l^2.$
Using this identity for each $\varpi\in N,$ the inverse Fourier transform for $h^{-q}K((Z_1 - Z_2)/h),$ the Fubini theorem and a change of variables we can write
\begin{eqnarray*}
I(h) &=& \int_{l ^2} \int_{\mathbb{R}^{q}}\mathbb{E}\left[ \left\langle
V_{1} e^{i\{t^{\top}Z_1 +  \left\langle x, W_{1}\right\rangle_{2}\}}, V_{2}
\; e^{-i\{t^{\top}Z_2 +  \left\langle x, W_{2}\right\rangle_{2}\}}
 \right\rangle_{\mathcal{H}_1} \right] \mathcal{F}[K](th)dt d\mu_Q(x)\\
&=&  \int_{l ^2} \int_{\mathbb{R}^{q}}\left\|\mathbb{E}\left[
V_{1} e^{i\{t^{\top}Z_1 +  \left\langle x, W_{1}\right\rangle_{2}\}}
 \right]  \right\|^2_{\mathcal{H}_1} \mathcal{F}[K](th)dtd\mu_Q(x),
\end{eqnarray*}
where $V_{j} = \mathbb{E}[U_j\mid Z_j,W_j]\omega(Z_j),$ $j=1,2.$  Deduce that
$$
 \mathbb{E}\left[\mathbb{E}[U\mid Z,W]\omega(Z)e^{i\{t^{\top}Z +  \left\langle x, W\right\rangle_{\mathcal{H}_2}\}}\right]=0,\qquad \forall t\in\mathbb{R}^q, \; x\in l^2.
$$
By the uniqueness of the Fourier Theorem in Hilbert spaces, see for instance Proposition 1.7 of \cite{DaPrato2006}, it follows that $\mathbb{E}[U\mid Z,W]=0.$
Now, the proof is complete.
\end{proof}

\qquad

\begin{lem_a}\label{vW_k} Suppose that Assumptions \ref{D}-(a) and \ref{K} are met.

(a)
\begin{equation*}
\sup_{t\in \mathbb{R}^q}  \frac{1}{n h^q} \sum_{i=1}^n K^k((t-Z_i)/h) = 
O_{\mathbb{P}}\left(  \sqrt{\frac{\ln n }{n h^q } }   \right) + o(h^{-q/2}), 
\end{equation*}
for $k=1$ or $k=2.$

(b) Let $0<\gamma_1,\gamma_2$  i.i.d. random variables such that $\mathbb{E}[\mathbb{E}^4(\gamma_1 \mid Z_1)]<\infty.$ Then
$
\mathbb{E}[\gamma_1 \gamma_2 h^{-q} K^2((Z_1-Z_2)/h)]
$ converges to a positive constant as $h\rightarrow 0.$
\end{lem_a}

\begin{proof}[Proof of Lemma \ref{vW_k}]
(a) We only consider the case $k=1$, the case $k=2$ is very similar.
By Theorem 2.1 of \cite{Vaart2011},
\begin{equation}\label{rate1}
\sup_{t\in \mathbb{R}^q} \left| \frac{1}{n} \sum_{i=1}^n \left\{K((t-Z_i)/h) -  \mathbb{E}[K((t-Z)/h) ] \right\} \right| = O_{\mathbb{P}} \left( \sqrt{\frac{h^{q} \ln n}{n}}\right)
\end{equation}
Indeed, let $\mathcal{G}$ be a class of functions of
the observations with envelope function $G,$ that here will is supposed bounded, and let
$$ J(\delta,\mathcal{G}, L^2 )=\sup_Q \int_0^\delta \sqrt{1+\ln N
  (\varepsilon \|G\|_{Q,2},\mathcal{G}, L^2(Q) ) } \; d\varepsilon
,\qquad 0<\delta\leq 1,
$$ denote the uniform entropy integral, where the supremum is taken
over all finitely discrete probability distributions $Q$ on the space
of the observations, and $\| G \|_{Q,2}$ denotes the norm of $G$ in
$L^2(Q)$. Let $Z_1,\cdots,Z_n$ be a sample of independent observations and let
\begin{equation*}
\mathbb{G}_n g=\frac{1}{\sqrt{n}}\sum_{i=1}^n \left\{g (Z_i)-\mathbb{E}[g(Z)]\right\} , \qquad \gamma \in\mathcal{G}
\end{equation*}
be the empirical process indexed by $\mathcal{G}$. If the covering number $N (\varepsilon ,\mathcal{G}, L^2(Q) ) $ is of polynomial order in
$1/\varepsilon,$ there exists a constant $c>0$ such that
$J(\delta,\mathcal{G}, L^2 )\leq c \delta \sqrt{\ln(1/\delta)}$ for
$0<\delta<1/2.$ Now if $\mathbb{E}g^2 < \delta^2 \mathbb{E}G^2$
for every $\gamma$ and some $0<\delta <1$,
Theorem 2.1 of \cite{Vaart2011} implies
\begin{equation}\label{vwww0}
 \sup_{\mathcal{G}}|\mathbb{G}_n g| = J(\delta,\mathcal{G}, L^2 )\left( 1 + \frac{ J(\delta,\mathcal{G}, L^2 )}{\delta^2 \sqrt{n} \|G\|_2 }\right) \|G\|_2 O_p(1),
\end{equation}
where $\|G\|_{2}^2 = \mathbb{E}G^2$ and the  $ O_p(1)$ term is independent of $n.$ Note that the family
$\mathcal{G}$ could change with $n$, as soon as the envelope is the
same for all $n$. We can thus apply this result to the family of functions
$\mathcal{G} = \{ K ((t - \; \cdot)/h) : t\in\mathbb{R}^q\}$ for a
sequence $h$ that converges to zero, the envelope
$G(\cdot)\equiv K(0)$, and $\delta = h^{q/2}.$ Its
entropy number is of polynomial order in $1/\varepsilon$,
independently of $n$, as $K(\cdot)$ is of bounded variation. Thus the rate in (\ref{rate1}) follows.

On the other hand, if $ |\mathcal{F}[f_Z](u)|^{2-a}$ is integrable for some $a\in(0,1],$ by the properties of the Fourier and inverse Fourier transforms, Fubini theorem and the Cauchy-Schwarz inequality, for any $t\in\mathbb{R},$
\begin{eqnarray}\label{ffrjj}
\mathbb{E}[h^{-q} K \left((t- Z)/h \right)] &=& \left|(2\pi)^{-q/2}\mathbb{E} \int_{\mathbb{R^q}} \exp\{i u^\top t \}\exp\{ -i u^\top Z\}\mathcal{F}[K] (hu) du\right|\notag\\
&=& \left| \int_{\mathbb{R^q}} \exp\{i u^\top t \}\mathcal{F}[f_Z](u)\mathcal{F}[K] (hu) du\right| \notag \\
&\leq & \left[ \int_{\mathbb{R^q}} |\mathcal{F}[f_Z](u)|^{2-a} du\right]^{\frac{1}{2-a}}
\left[ \int_{\mathbb{R}} |\mathcal{F}[K](hu)|^{(2-a)/(1-a)} du\right]^{\frac{1-a}{2-a}}\notag\\
&\leq& C
\left[ h^{-q}\int_{\mathbb{R}^q} |\mathcal{F}[K](u)| du\right]^{\frac{1-a}{2-a}}\notag\\
&=&
Ch^{-q(1-a)/(2-a)}\notag\\
&=& o(h^{-q/2}),
\end{eqnarray}
for some constant $C$  independent of $t.$ Alternatively, if the density $f_Z$ is bounded, by a change of variables we can write
\begin{equation}\label{ffrjj_2}
\mathbb{E}[h^{-q} K \left((t- Z)/h \right)] = \int_{\mathbb{R}^q} K(u) f_Z(t-uh)du \leq C^\prime ,
\end{equation}
for some constant $C^\prime$ independent of $t.$ From equations (\ref{spectral}), (\ref{rate1}), (\ref{ffrjj}) and (\ref{ffrjj_2})
\begin{equation}\label{bdd1}
\sup_{t\in \mathbb{R}^q}  \frac{1}{nh^q} \sum_{i=1}^n K^k((t-Z_i)/h) = 
O_{\mathbb{P}}\left(  \sqrt{\frac{\ln n }{n h^q } }   \right) + o(h^{-q/2}),
\end{equation}
for $k=1$ or $k=2.$

(b) Let $e(z)=\mathbb{E}[\gamma_1 \mid Z_1 = z].$ If $f_Z$ satisfies the condition $\int_{\mathbb{R}^q} |\mathcal{F}[f_Z]|^{2-a}(t) dt <\infty$ for some $a\in (0,1]$, then $\int_{\mathbb{R}^q} f_Z^2 <\infty$ and hence by Cauchy-Schwarz inequality $$\int_{\mathbb{R}^q} f_Z^2 e^2 \leq \left( \int_{\mathbb{R}^q} f_Z^2  \right)^{1/2} \left( \int_{\mathbb{R}^q} f_Z e^4  \right)^{1/2}<\infty.$$  Using Fubini Theorem, the inverse Fourier Transform formula and Parseval identity, we can write
\begin{eqnarray*}
\mathbb{E}[\gamma_1 \gamma_2 h^{-q} K \left((Z_1- Z_2)/h \right)] &=& \mathbb{E}[h^{-q} e(Z_1)e(Z_2) K \left((Z_1- Z_2)/h \right)]\\
&=& (2\pi)^{-\frac{q}{2}}\, \mathbb{E} \int_{\mathbb{R^q}}  e(Z_1)\exp\{i u^\top Z_1 \} \\
&&\qquad \qquad \qquad \times e(Z_2) \exp\{ -i u^\top Z_2\}\mathcal{F}[K] (hu) du\notag\\
 &=& \int_{\mathbb{R^q}} \left|\mathcal{F}[f_Z e](u)\right|^2 |\mathcal{F}[K]| (hu) du \notag \\
  &\rightarrow& \int_{\mathbb{R^q}} \left|\mathcal{F}[f_Ze](u)\right|^2  du \notag\\
&=&\int_{\mathbb{R^q}} f_Z^2(u)e^2(u) du ,
\end{eqnarray*}
where for the limit we use the Dominated Convergence Theorem. If $f_Z$ is bounded, we can use a change of variables like for equation (\ref{ffrjj_2}) and again the Dominated Convergence Theorem to obtain the same strictly positive and finite limit.
\end{proof}

\qquad

\begin{proof}[Proof of Theorem \ref{cons_Tn}]
The proof is based on
the Central Limit Theorem 5.1 of \cite{Jong1987}.
Let
$$
\Omega_{ij} = \frac{1}{n(n-1)h^{q}}\left\langle U_{i}\omega(Z_i), \;U_{j}\omega (Z_j) \right\rangle_{\mathcal{H}_1} K_{ij}(h)\varphi_{ij}
,\quad 1\leq i\neq j\leq n,
$$
and $\Omega_{ii} =0,$ $1\leq i\leq n.$
Let  $\Omega(n)=\sum_{i\neq j} \Omega_{ij}$ and $\sigma(n)^2 = 2  \sum_{i\neq j} \sigma^2_{ij}$ where
$$
\sigma_{ij}^2 =  \mathbb{E} [ \Omega_{ij} ^2\mid (Z_i,W_i),(Z_j,W_j)] = \frac{V_{ij}^2K^2_{ij}(h)\varphi_{ij}^2}{n^2(n-1)^2h^{2q}}
$$
and
$$
V_{ij}^2 = \mathbb{E} [ \left\langle U_{i}\omega(Z_i), \;U_{j}\omega (Z_j) \right\rangle_{\mathcal{H}_1}^2 \mid (Z_i,W_i),(Z_j,W_j)]
$$
Consider the following conditions:
\begin{enumerate}
\item there exists a sequence of real numbers $k_n$ such that
\begin{equation}\label{clt1}
k^2_n \sigma(n)^{-2} \max_{1\leq i\leq n} \sum_{1\leq j\leq n}  \sigma^2_{ij} = o_{\mathbb{P}}(1)
\end{equation}
and
\begin{equation}\label{clt2}
\max_{1\leq i\neq j \leq n}   \sigma^{-2}_{ij} \mathbb{E}\left[ \Omega_{ij}^2 \mathbf{1}_{\{|\Omega_{ij}|> k_n \sigma_{ij}\}}\mid (Z_i,W_i),(Z_j,W_j)\right] = o_{\mathbb{P}}(1);
\end{equation}
\item
\begin{equation}\label{clt3}
\sigma(n)^{-2} \max_{1\leq i\leq n}  \mu^2_{i} = o_{\mathbb{P}}(1),
\end{equation}
where $\mu_1,\cdots,\mu_n$ are the eigenvalues of the matrix $(\sigma_{ij})$.
\end{enumerate}
If these conditions hold true, using the characterization of the convergence in probability based on almost surely convergence subsequences, Theorem 5.1 of \cite{Jong1987} applied conditionally on the covariates implies that  for any $t\in\mathbb{R},$
$$
\mathbb{P}\left(\sigma(n)^{-1}\Omega(n) \leq t \mid (Z_1,W_1),\cdots,(Z_n,W_n)  \right) - \Phi(t) = o_{\mathbb{P}}(1).
$$
By the dominated convergence theorem, $\sigma(n)^{-1}\Omega(n)$ converges to in law to a standard normal distribution. Hence, it remains to check conditions (\ref{clt1}) to (\ref{clt3}).

First, let us bound from below $\sigma(n).$ By Assumption \ref{D}-(c)-(i),
$
V_{ij}^2 \geq \underline{\sigma}^{2}
$
almost surely, so that
$$
\sigma(n)^2 \geq  \frac{\underline{\sigma}^{2}}{n^2(n-1)^2h^{2q}}\sum_{i\neq j} K^2_{ij}(h)\varphi^2_{ij}\geq  \frac{\underline{\sigma}^{2}}{n^2(n-1)^2h^{2q}}
\sum_{i\neq j} K^2_{ij}(h)\lambda^2_{i}\lambda^2_{j} = \underline{\sigma}^{2}\tau(n)^2,$$
where
$\lambda_i = \exp(-2\|W_i\|^2_{\mathcal{A}}).$
By  standard calculations, the variance
$n(n-1)h^{q}\tau(n)^2$ tends to zero. By Lemma \ref{vW_k}-(b) the expectation of
$n(n-1)h^{q}\tau(n)^2$  tends to a positive constant. Deduce that
\begin{equation}\label{rate01}
\sigma(n)^{-2} =O_{\mathbb{P}}(n^2h^{q}).
\end{equation}
Next, note that by H\"{o}lder inequality and Assumption \ref{D}-(c)-(ii),
$$
V_{ij}^2 \leq  \mathbb{E} [ \left\| U_{i}\omega(Z_i) \right\|_{\mathcal{H}_1}^2 \mid (Z_i,W_i),(Z_j,W_j)]\mathbb{E} [ \left\| U_{j}\omega(Z_j) \right\|_{\mathcal{H}_1}^2 \mid (Z_i,W_i),(Z_j,W_j)]\\
\leq  C^{4/\nu},
$$
almost surely. Deduce from this and Lemma \ref{vW_k}-(a) that
\begin{equation}\label{rate01b}
\max_{1\leq i \leq n} \sum_{1\leq j\leq n }\sigma_{ij}^2 =o_{\mathbb{P}}(n^{-2}h^{-q}).
\end{equation}
Then condition (\ref{clt1}) follows from (\ref{rate01}) and (\ref{rate01b})
for some suitable sequence $k_n\rightarrow\infty.$

Next, let us note that
$$
\frac{\Omega^2_{ij}}{\sigma_{ij}^2} \leq \underline{\sigma}^{-2}\left\langle U_{i}\omega(Z_i), \;U_{j}\omega (Z_j) \right\rangle_{\mathcal{H}_1}^2,
$$
so that for any $i$ and $j,$ by H\"{o}lder inequality, Markov inequality and Assumption \ref{D}-(c)-(ii),
\begin{multline*}
\sigma^{-2}_{ij} \mathbb{E}\left[ \Omega_{ij}^2 \mathbf{1}_{\{|\Omega_{ij}|> k_n \sigma_{ij}\}}\mid (Z_i,W_i),(Z_j,W_j)\right] \leq
\underline{\sigma}^{-2}
\mathbb{E}\left[ \left\langle U_{i}\omega(Z_i), \;U_{j}\omega (Z_j) \right\rangle_{\mathcal{H}_1}^2 \right.\\
\qquad \qquad \qquad \qquad \qquad \qquad \qquad \left. \times \mathbf{1}_{\{|\left\langle U_{i}\omega(Z_i), \;U_{j}\omega (Z_j) \right\rangle_{\mathcal{H}_1}|> \underline{\sigma} \;k_n \}}\mid (Z_i,W_i),(Z_j,W_j)\right]\\
\leq \underline{\sigma}^{-2} \mathbb{E}^{2/\nu} [ \left\| U_{i}\omega(Z_i) \right\|_{\mathcal{H}_1}^\nu
\left\| U_{j}\omega(Z_j) \right\|_{\mathcal{H}_1}^\nu
\mid (Z_i,W_i),(Z_j,W_j)]\\
\qquad \qquad \qquad\qquad \times \mathbb{P}^{(\nu-2)/\nu} [ |\left\langle U_{i}\omega(Z_i), \;U_{j}\omega (Z_j) \right\rangle_{\mathcal{H}_1}|> \underline{\sigma} \;k_n  \mid (Z_i,W_i),(Z_j,W_j)]\\
\leq\underline{\sigma}^{-2} C^{4/\nu}[C^{2/\nu}(\underline{\sigma} \;k_n)^{-1}]^{(\nu-2)/\nu},
\end{multline*}
almost surely. Thus condition (\ref{clt2}) holds true for any $k_n\rightarrow \infty.$

To check condition (\ref{clt3}), let $\mathcal{K} $ be the matrix with elements
\begin{equation}\label{matrix_k}
\mathcal{K}_{ij}=  V_{ij}\varphi_{ij}K\left( (Z_i-Z_j)/h \right)/[n(n-1)h^q], \quad i\neq j, \quad \text{and} \quad
\mathcal{K}_{ii}= 0,
\end{equation}
and $\|| \mathcal{K} \||_2$  is the spectral norm of $\mathcal{K}.$
By definition, $\|| \mathcal{K} \||_2 =\sup_{u\in\mathbb{R}^n , u\neq 0}{\Vert
\mathcal{K} u\Vert }/{\Vert u\Vert }$ and $|u^\prime \mathcal{K} w|\leq \|| \mathcal{K} \||_2 \|u\| \|w\|$ for any $u,w\in\mathbb{R}^n.$
By Cauchy-Schwarz inequality, for any $u\in
\mathbb{R}^{n}$,
\begin{eqnarray}
\left\Vert \mathcal{K} u\right\Vert ^{2} &=& \sum_{i=1}^{n}\left(
\sum_{j=1,j\neq i}^{n}V_{ij}\frac{K\left( (Z_i-Z_j)/nh \right) }{h^q\,n(n-1)}u_{j}\right) ^{2}\notag \\
&\leq &C^{4/\nu} \sum_{i=1}^{n}\left( \sum_{j=1,j\neq i}^{n}\frac{K\left(
(Z_i-Z_j)/nh\right) }{h^q\,n(n-1)}\right)
\sum_{j=1,j\neq i}^{n}\frac{K\left( (Z_i-Z_j)/nh\right) }{h^q\,n(n-1)}\,u_{j}^{2} \notag \\
&\leq &C^{4/\nu} \left\Vert u\right\Vert ^{2}n^{-2} \left[ \max_{1\leq i\leq n}\left(
\sum_{j=1,j\neq i}^{n}\frac{K\left( (Z_i-Z_j)/nh\right) }{h^q\,(n-1)}\right) \right] ^{2},\label{spectral}
\end{eqnarray}
for some constant $c>0$. By Lemma \ref{vW_k}-(a) deduce that
$$
 \max_{1\leq i\leq n}  \mu^2_{i} = \frac{1}{n^2} \left[ O_{\mathbb{P}}\left(  \frac{\ln n }{n h^q }    \right) + o(h^{-q}) \right]= o_{\mathbb{P}}\left(  \frac{1 }{n^2 h^q }    \right) .
$$
Condition (\ref{clt3}) follows from this and the rate (\ref{rate01}). Now the proof is complete. \end{proof}

\qquad

\begin{proof}[Proof of Theorem \ref{consist}]
Let us simplify the notation and denote $\omega_i = \omega(Z_i)$ and $\delta_i = \delta(Z_i,W_i).$ Next let us decompose
\begin{eqnarray*}
I_{n}(h) &=& \frac{1}{n(n-1)h^{q}}\sum\limits _{1\leq i\neq j\leq n}\left\langle U_{i}^0 \omega_{i}, \;U_{j}^0 \omega_{j} \right\rangle_{\mathcal{H}_1} K_{ij}(h)\;
\varphi_{ij}\\
&+&\frac{r_n}{n(n-1)h^{q}}\sum\limits _{1\leq i\neq j\leq n}\left\langle U_{i}^0 \omega_{i}, \;\delta_{j} \omega_{j} \right\rangle_{\mathcal{H}_1} K_{ij}(h)\;
\varphi_{ij}\\
&+& \frac{r_n^2}{n(n-1)h^{q}}\sum\limits _{1\leq i\neq j\leq n}\left\langle \delta_{i} \omega_{i}, \;\delta_{j} \omega_{j} \right\rangle_{\mathcal{H}_1} K_{ij}(h)\varphi_{ij}\\
&=& I_{0n}+ 2I_{1n}+I_{2n}.
\end{eqnarray*}
The rate of $I_{0n}$ is given by Theorem \ref{cons_Tn}, so that it remains to investigate the rates of $I_{1n}$ and $I_{2n}$ and to bound in probability $v^2_{n}(h).$ By standard calculations,
$$
\mathbb{E}(I_{0n}) = \mathbb{E}(I_{1n}) =0, \quad Var(  I_{0n})=O_{\mathbb{P}}(n^{-2}h^{-q}) \quad \text{and} \quad
 Var(  I_{1n})=O_{\mathbb{P}}(r_n^2 n^{-1}h^{-q}).
$$
Moreover,
$$
\mathbb{E}(I_{2n}) = r_n^2 \; \mathbb{E}[\left\langle \delta_{1} \omega_{1}, \;\delta_{2} \omega_{2} \right\rangle_{\mathcal{H}_1}h^{-q} K_{12}(h)\varphi_{12}] \quad \text{and} \quad Var(  I_{1n})=O_{\mathbb{P}}(r_n^4 n^{-1}h^{-q}).
$$
By dominated convergence we have
\begin{eqnarray*}
\mathbb{E}[\left\langle \delta_{1} \omega_{1}, \;\delta_{2} \omega_{2} \right\rangle_{\mathcal{H}_1}h^{-q} K_{12}(h)\varphi_{12}]&=& \mathbb{E}[\mathbb{E}[\left\langle \delta_{1} \omega_{1}, \;\delta_{2} \omega_{2} \right\rangle_{\mathcal{H}_1}\varphi_{12}\mid Z_1,Z_2]h^{-q} K_{12}(h)]\\
&=&(2\pi)^{-\frac{q}{2}}\, \mathbb{E} \left[\int_{\mathbb{R^q}} \mathbb{E}[\left\langle \delta_{1} \omega_{1}, \;\delta_{2} \omega_{2} \right\rangle_{\mathcal{H}_1}\varphi_{12}\mid Z_1,Z_2] \right.\\
&&\qquad  \left.\frac{}{}\times \exp\{i u^\top Z_1 \} \exp\{ -i u^\top Z_2\} \mathcal{F}[K] (hu) du\right]\notag\\
&\rightarrow& (2\pi)^{-\frac{q}{2}}\, \int_{\mathbb{R^q}} \mathbb{E}\left[\left\langle \delta_{1} \omega_{1}, \;\delta_{2} \omega_{2} \right\rangle_{\mathcal{H}_1}\varphi_{12}\right. \\
&&\qquad  \qquad\qquad \quad\left.\times \exp\{i u^\top Z_1 \} \exp\{ -i u^\top Z_2\}\right] du\notag
\end{eqnarray*}
By arguments as used in the proof of Lemma \ref{lemma_fund} the expectation of $I_{2n}$ could be shown to be strictly positive. Since $r_{n}^{2}nh^{q/2}\rightarrow\infty,$ the result follows.
\end{proof}

\qquad

\begin{lem}\label{diff_l2b}
Suppose the assumptions of Theorem \ref{test_estim}  hold true.
Then, for $l=1$ and $l=2$,
\begin{equation}\label{sn_spec_2}
\max_{1\leq i\leq  n } \frac{1}{nh}\sum_{j=1}^n \left| \; [\widehat K_{ij} (h) \widehat \varphi_{ij}]^l   - [ K_{ij} (h)  \varphi_{ij}]^l \right|  = o_{\mathbb{P}}(1)
\end{equation}
and
\begin{equation}\label{sn_spec5_2}
\frac{1}{n(n-1)h}\sum_{1\leq i\neq j \leq n} \left\{ [\widehat K_{ij} (h) \widehat \varphi_{ij}]^2   - [  K_{ij} (h)  \varphi_{ij}]^2\right\} = o_{\mathbb{P}}(1).
\end{equation}
\end{lem}

\begin{proof}[Proof of Lemma \ref{diff_l2b}]
Assume that $nh^{4}/\ln^{2} n \rightarrow \infty$. By the Lipschitz property of the kernel and of the $\varphi(\cdot)$ function,  the bound (\ref{approx_cond}) and conditions (\ref{exp_cond}), for $l=1$ and $l=2,$
$$
\max_{i,j}h^{-1}\left|
[\widehat K_{ij} (h) \widehat \varphi_{ij}]^l   - [ K_{ij} (h)  \varphi_{ij}]^l
\right|
 \leq   C h^{-2} \Delta_n   \max_{1\leq i \leq n} |\Gamma_i|= O_{\mathbb{P}}(n^{-1/2} h^{-2}\ln n )= o_{\mathbb{P}}(1).
$$
The rates (\ref{sn_spec_2}) and (\ref{sn_spec5_2}) follow.

If conditions at point (2) or point (3) are met, the arguments are of different nature. First, note that the conditions of point (3) involve that $f_Z$ is bounded. Next, since the kernel $K$ is of bounded univariate kernels, let $K_1$ and $K_2$ non decreasing bounded functions such that $K^l = K_1 - K_2$ and denote $K_{1h}=K_1(\cdot/h)$. Clearly, it is sufficient to prove (\ref{sn_spec_2})  for $K_1$, similar arguments apply for $K_2$ and hence we get the results for $K^l$. For simpler writings let us assume that $K^l$ is differentiable and let $K_1(x) = \int^x_{-\infty} |K^{l\;\prime} (t) | dt,$ $x\in\mathbb{R}.$ The general case of a bounded variation $K^l$ can be handled with obvious modifications. Let $b_n\downarrow 0$ such that $b_n \sqrt{n}/\ln n \rightarrow\infty$
and define the event
\begin{equation}\label{set_en}
\mathcal{E}_{1n} = \{\max_{1\leq i\leq n} [\|\widehat Z_i - Z_i \| + \|\widehat W_i - W_i \|_{\mathcal{H}_2}|] \leq b_n\},
\end{equation}
so that $
\mathbb{P}(\mathcal{E}^c_{1n} )\rightarrow 0.
$
Since $|\exp(-t^2/2) - \exp(-s^2/2)|\leq |t-s|,$ on the set $\mathcal{E}_{1n}$, $\forall i,j$
\begin{eqnarray*}
&&\hspace{-2.5cm}-  \;b_n K_{1h}( Z_{i} -  Z_{j} + 2b_n )
-[  K_{1h}( Z_{i} -  Z_{j} - 2b_n ) - K_{1h}\left(Z_{i} - Z_{j}  \right)] \varphi_{ij}\\
& \leq & \left|K_{1h}\left(\widehat Z_{i} - \widehat Z_{j}  \right) - K_{1h}\left( Z_{i} -  Z_{j}  \right) \right| \varphi_{ij} -K_{1h}\left(\widehat Z_{i} - \widehat Z_{j}  \right) \left|\varphi_{ij} - \widehat \varphi_{ij}\right|\\
& \leq & \left|K_{1h}\left(\widehat Z_{i} - \widehat Z_{j}  \right)\widehat \varphi_{ij} - K_{1h}\left( Z_{i} -  Z_{j}  \right) \varphi_{ij} \right|.
\end{eqnarray*}
Similarly
\begin{eqnarray*}
&&\hspace{-2cm}\left|K_{1h}\left(\widehat Z_{i} - \widehat Z_{j}  \right)\widehat \varphi_{ij} - K_{1h}\left( Z_{i} -  Z_{j}  \right) \varphi_{ij} \right|\\
&\leq &  \;b_n K_{1h}( Z_{i} -  Z_{j} + 2b_n ) +
[  K_{1h}( Z_{i} -  Z_{j} + 2b_n ) - K_{1h}\left(Z_{i} - Z_{j}  \right)] \varphi_{ij}.
\end{eqnarray*}
We focus on the second inequality, the first one can be handled similarly.  To justify (\ref{sn_spec_2}) we can write
\begin{alignat*}{1}
 & \frac{1}{nh}\sum_{j=1}^n \left|K_{1h}\left(\widehat Z_{i} - \widehat Z_{j}  \right)\widehat \varphi_{ij} - K_{1h}\left( Z_{i} -  Z_{j}  \right) \varphi_{ij} \right|\\
\leq \, & \frac{ b_n}{h} \mathbb{E}[K_{1h}\left( Z_{i} -  Z   + 2 b_n \right) \mid Z_i]
\\
& +\frac{ b_n}{nh}\sum_{j=1}^n \left[ K_{1h}\left( Z_{i} -  Z_{j}  + 2 b_n \right) - \mathbb{E}[K_{1h}\left( Z_{i} -  Z_{j}  + 2 b_n \right) \mid Z_i]\right]\\
& + \frac{1}{nh}\sum_{j=1}^n \{ K_{1h}\left(Z_{i} -  Z_{j} + 2b_n  \right)\varphi_{ij} -
\mathbb{E} [K_{1h}\left(Z_{i} -  Z_{j} +2b_n \right)\varphi_{ij}\mid Z_i, W_i]\}\\
& - \frac{1}{nh}\sum_{j=1}^n \{ K_{1h}\left(Z_{i} -  Z_{j}   \right) \varphi_{ij}-
\mathbb{E}[ K_{1h}\left(Z_{i} -  Z_{j} \right)\varphi_{ij}\mid Z_i, W_i]\}\\
& +
\mathbb{E}[ h^{-1} K_{1h}\left(Z_{i} -  Z_{j} +2b_n \right)\varphi_{ij} \mid Z_i, W_i]
 - \mathbb{E}[h^{-1} K_{1h}\left(Z_{i} -  Z_{j}  \right) \varphi_{ij} \mid Z_i, W_i]\\
= \, & A_{0n,i}+A_{1n,i}+A_{2n,i}-A_{3n,i}+A_{4n,i}.
\end{alignat*}
Since $f_Z$ is supposed bounded, by a simple change of variable, we get $\max_{1\leq i\leq n} A_{0n,i}=O(b_n)=o(1)$.

The terms $A_{1n,i},$ $A_{2n,i}$ and $A_{3n,i}$ could be treated similarly, hence we only investigate  $A_{2n,i}.$ First note that, since $\varphi_{ij}\leq 1,$ the function $K_1$ is bounded and integrable and $f_Z$ is bounded, for all $j$ we have
$$
\mbox{\rm Var} \left(K_{1h}\left(Z_{i} -  Z_{j} + 2b_n  \right)\varphi_{ij}  \mid Z_i, W_i\right)\leq Ch,\quad \forall i,
$$
for some constant $C$ independent of $n$ and $Z_i,W_i.$ Using this conditional variance bound and applying Bernstein inequality\footnote{Recall that Bernstein inequality states that if $W_1,\cdots,W_n$ are i.i.d. centered random  variables of variance $\sigma^2$ taking values in the interval $[-M,M],$ then for any $s>0$
\begin{equation*}
\mathbb{P}\left(\left|\frac{1}{n}\sum _{i=1} ^n W_i \right| >s\right) \leq 2\exp\left( -\; \frac{n s^2}{2[\sigma^2 + M s/3]} \right).
\end{equation*}
}
conditionally on the $Z_i,W_i$'s, for any
$t>0$
\begin{multline*}
\mathbb{P}\left[ \max_{1\leq i\leq n}\left\vert \frac{1}{n}\sum_{j=1}^n
\left\{K_{1h}\left( Z_{i} - Z_{j} + 2b_n  \right) -\esp\left[
K_{1h}\left( Z_{i} - Z_{j} + 2b_n \right) |X_{i},W_i\right]\right\}
\right\vert > th\right]  \\
\leq \sum_{i=1}^n\mathbb{E}\!\left[\! \mathbb{P}\!\left[ \left\vert
\frac{1}{n}\sum_{j=1}^nK_{1h}\left( Z_{i} - Z_{j} + 2b_n
\right)   -\mathbb{E}\left[
K_{1h}\left( Z_{i} - Z_{j} + 2b_n \right)
|X_{i},W_i\right]\right\vert > th  \mid X_{i},W_i\right] \right]  \\
\leq 2n\exp \left( -\frac{t^{2}}{2}\; \frac{nh^{2}}{
Ch+thM/3}\right) \leq 2\exp \left( [\ln n]\left[1-\frac{t^{2}}{2}\; \frac{nh/\ln n }{
C+tM/3}\right]\right)
\rightarrow 0\;,
\end{multline*}%
since $nh /\ln n \rightarrow\infty $ under the conditions of point (2) or those of point (3) (here $M$ is any constant that bounds $K_1$). Deduce that $\max_{1\leq i\leq n} A_{2n,i}=o_{\mathbb{P}}(1).$

To complete the proof of (\ref{sn_spec_2}) it remains to investigate the convergence of $A_{4n,i}$ uniformly with respect to $i.$ First, since $\varphi_{ij}\leq 1,$
\begin{multline*}
\left|\mathbb{E}[ h^{-1}K_{1h}\left(Z_{i} - Z_j  +2b_n \right)\varphi_{ij}\mid Z_i,W_i]
- \mathbb{E}[h^{-1} K_{1h}\left(Z_{i} - Z_j  \right)\varphi_{ij}\mid Z_i,W_i]\right| \\ \leq \mathbb{E}\left[ \left|h^{-1}K_{1h}\left(Z_{i} - Z_j  +2b_n \right)
- h^{-1} K_{1h}\left(Z_{i} - Z_j  \right)\right|\mid Z_i\right].
\end{multline*}
If the conditions of point (2) are met, suppose that the sequence $b_n$ used for the definition of the set $\mathcal{E}_{1n}$  in equation (\ref{set_en}) is such that $b_n/h \rightarrow 0$ and $b_n \sqrt{n}/ \ln n \rightarrow\infty.$ By obvious calculations and changes of variables and using the uniform bound for $f_Z,$ for each $1\leq i\leq n$
\begin{multline}\label{int_df}
\mathbb{E}\left[ \left|h^{-1}K_{1h}\left(Z_{i} - Z_j  +2b_n \right)
- h^{-1} K_{1h}\left(Z_{i} - Z_j  \right)\right|\mid Z_i\right]\\
= \int_{\mathbb{R}}[K_1 (u + 2b_nh^{-1})- K_1 (u)]  f_Z (Z_i -uh) du\\\leq
\int_{\mathbb{R}} \left[\int_ 0^{2b_nh^{-1}}|K^\prime (u+v)|dv \right] f_Z (Z_i -uh) du\\
\leq C\int_ 0^{2b_nh^{-1}}\left[\int_{\mathbb{R}} |K^\prime (u+v)|du \right] dv \rightarrow 0.
\end{multline}
Finally, if the conditions of point (3) are met, by an alternative but still obvious change of variables and using the uniform continuity of $f_Z$ and the fact that $b_n\rightarrow 0$, for each $1\leq i\leq n$
\begin{multline*}\label{int_df2}
\mathbb{E}\left[ \left|h^{-1}K_{1h}\left(Z_{i} - Z_j  +2b_n \right)
- h^{-1} K_{1h}\left(Z_{i} - Z_j  \right)\right|\mid Z_i\right]\\ \leq \int_{\mathbb{R}}K_1 (u)\left|f_Z (2b_n + Z_i -uh) -  f_Z (Z_i -uh)\right| du\\
\leq \sup_{t\in\mathbb{R}}\left|f_Z (2b_n + t) -  f_Z (t)\right| \int_{\mathbb{R}}K_1 (u) du
\rightarrow 0.
\end{multline*}
Now the arguments are complete for justifying  (\ref{sn_spec_2}). The rate in (\ref{sn_spec5_2}) could be easily derived from  the rate in (\ref{sn_spec_2}).
\end{proof}

\begin{figure}
\begin{centering}
\includegraphics[scale=0.5,angle=270]{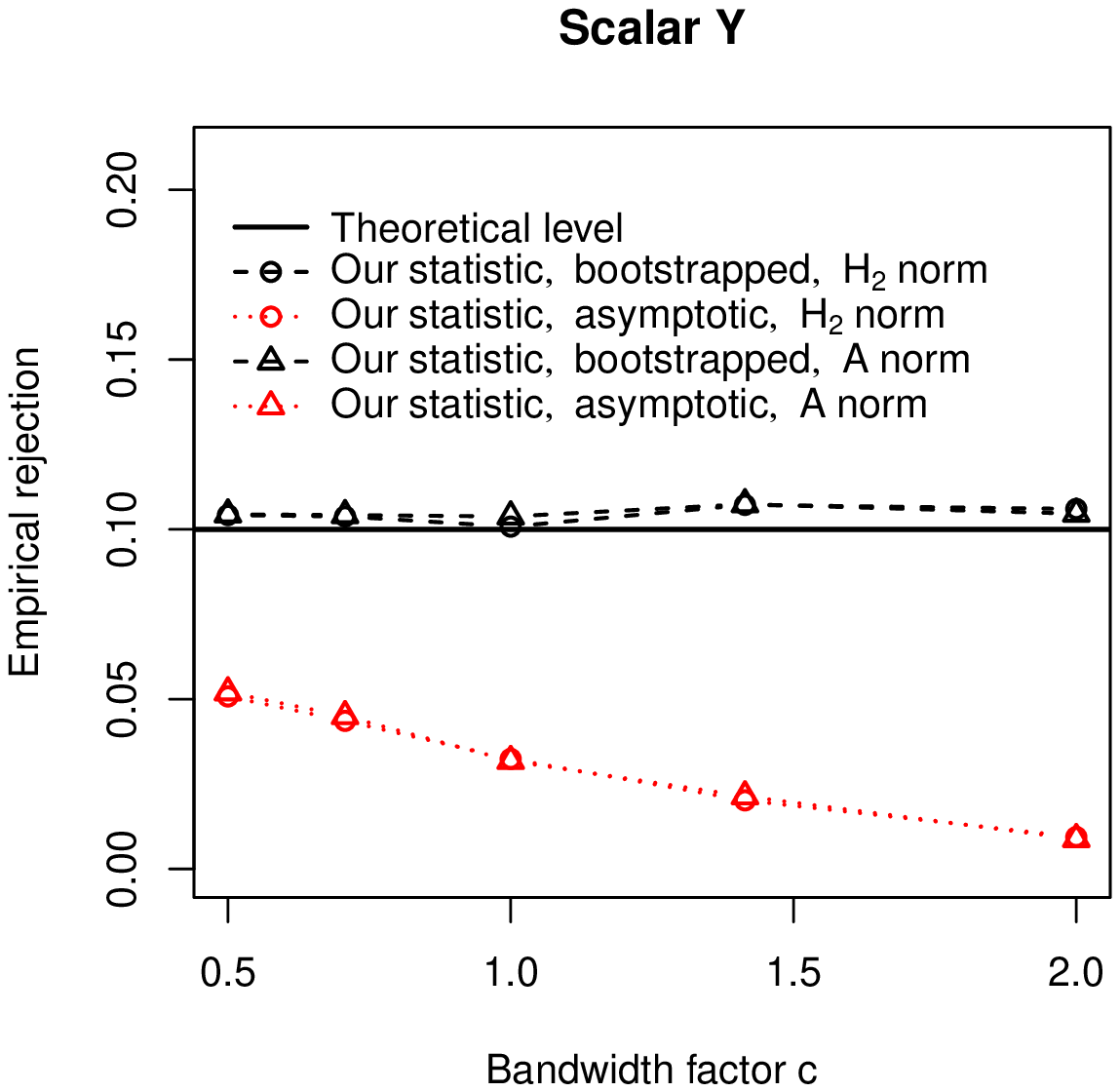}\includegraphics[scale=0.5,angle=270]{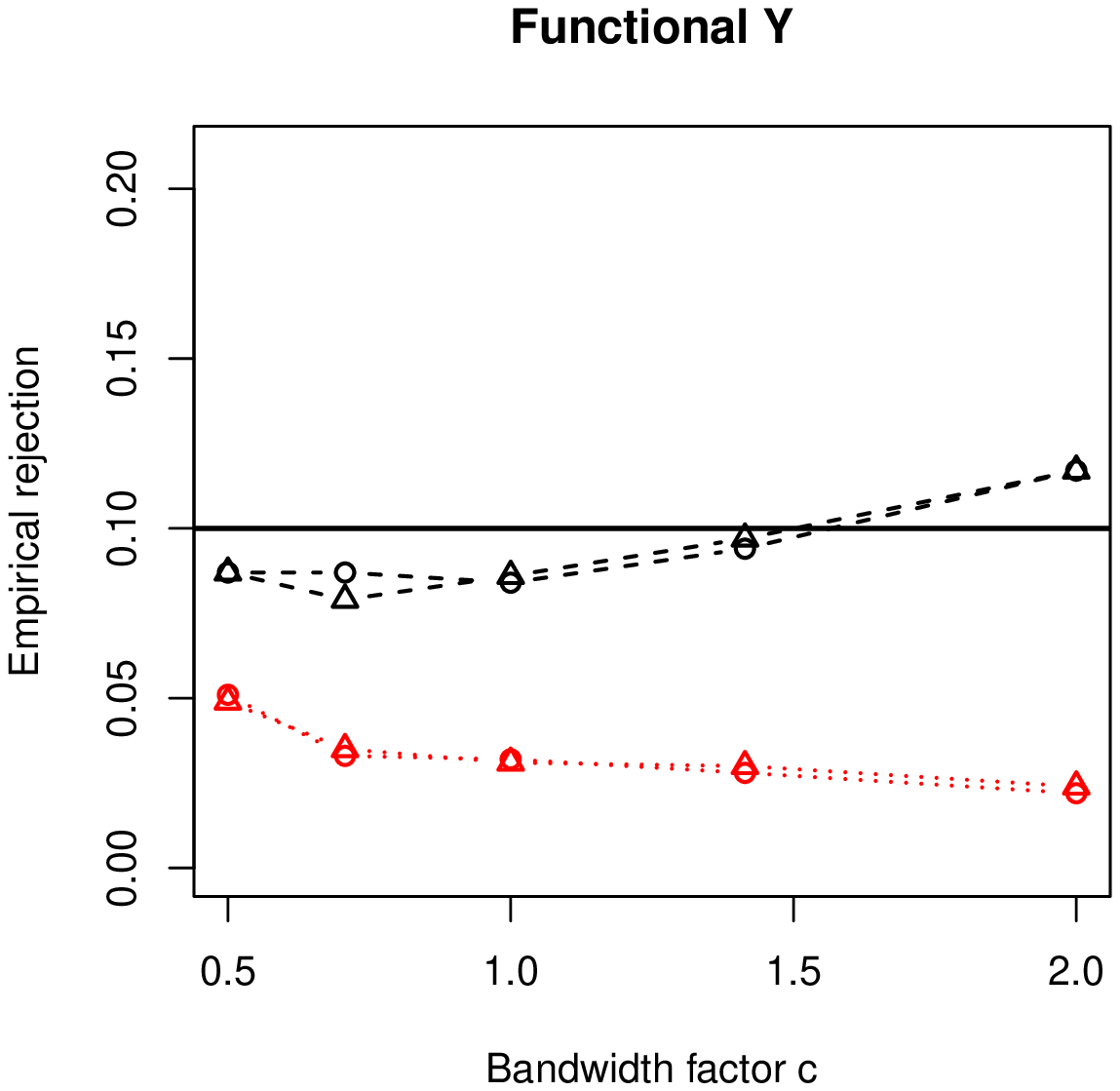}

\caption{Empirical rejection for scalar (left panel) and functional (right panel) $Y$ under the null hypothesis.\label{fig:Level}}
\end{centering}
\end{figure}

\begin{figure}
\begin{centering}
\includegraphics[scale=0.5,angle=270]{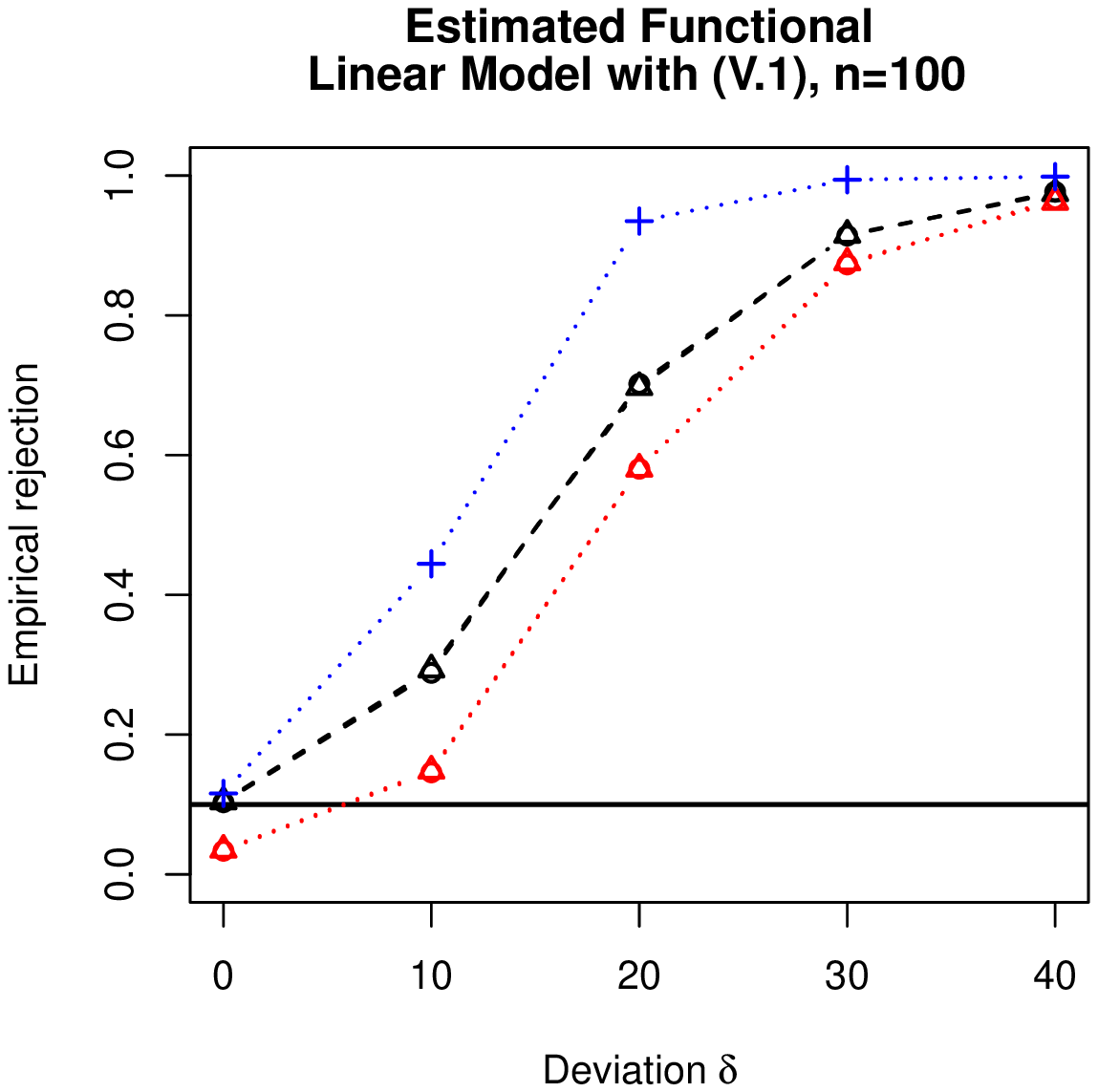}\includegraphics[scale=0.5,angle=270]{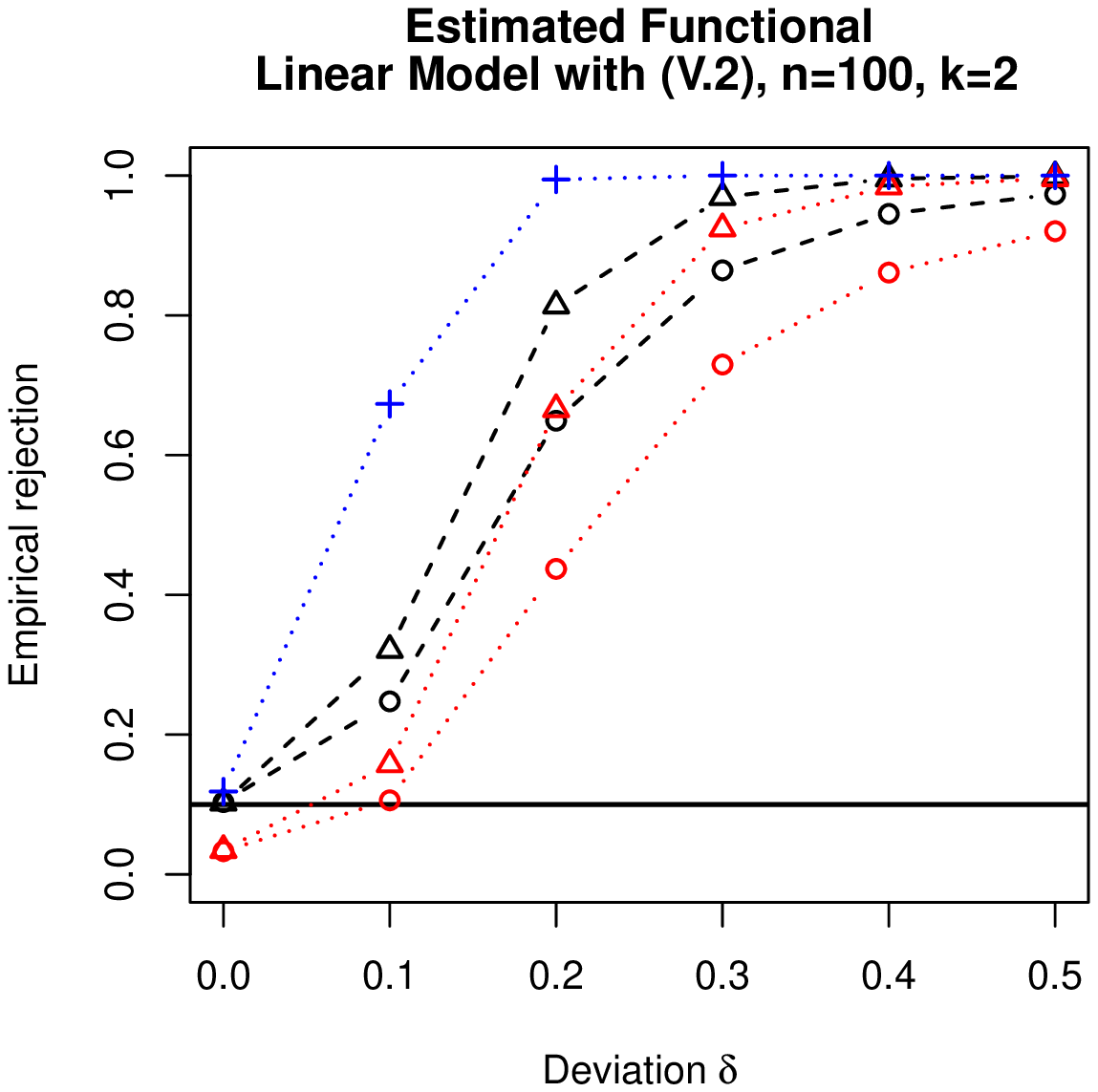}
\par\end{centering}

\begin{centering}
\includegraphics[scale=0.5,angle=270]{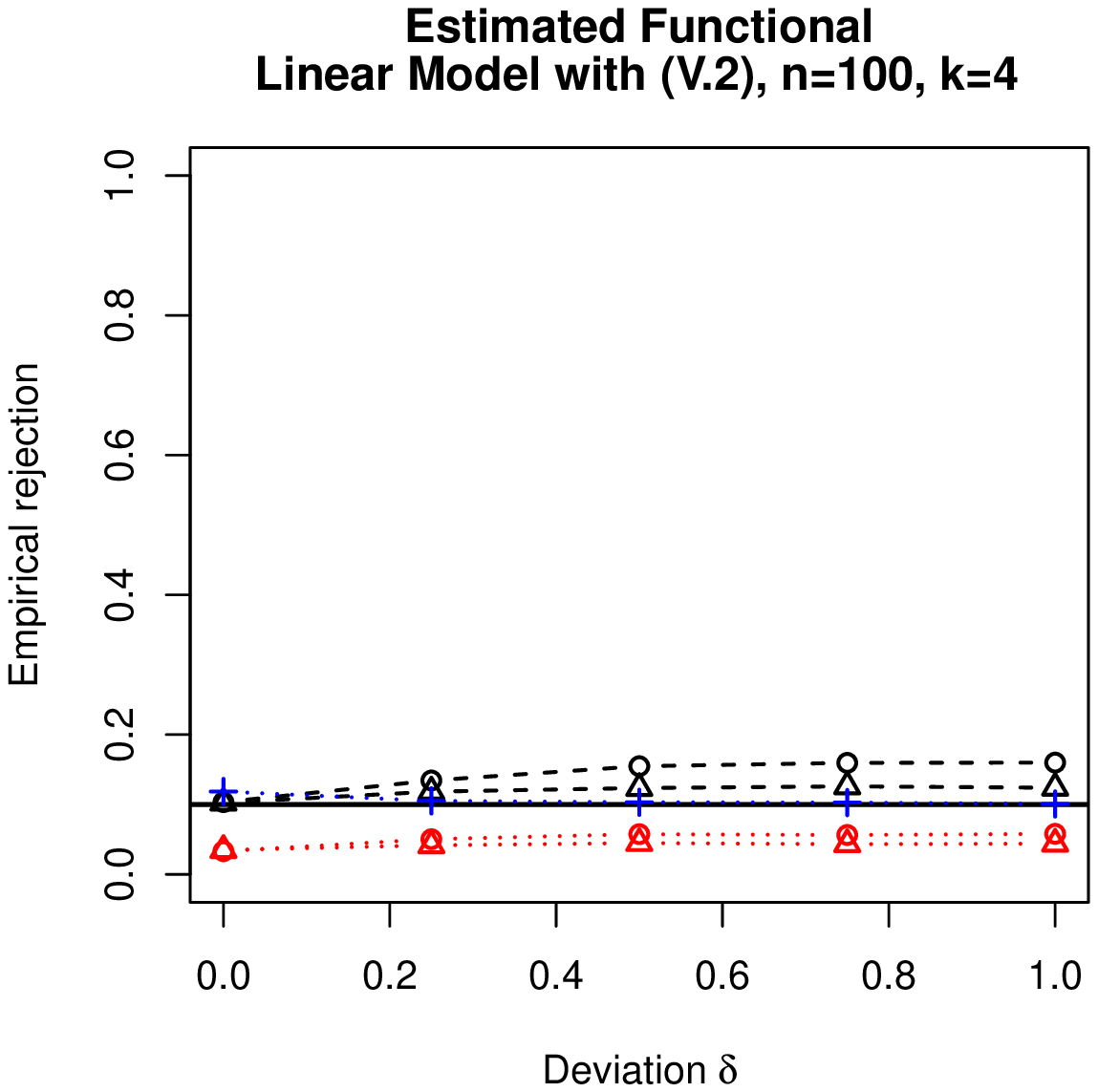}\includegraphics[scale=0.5,angle=270]{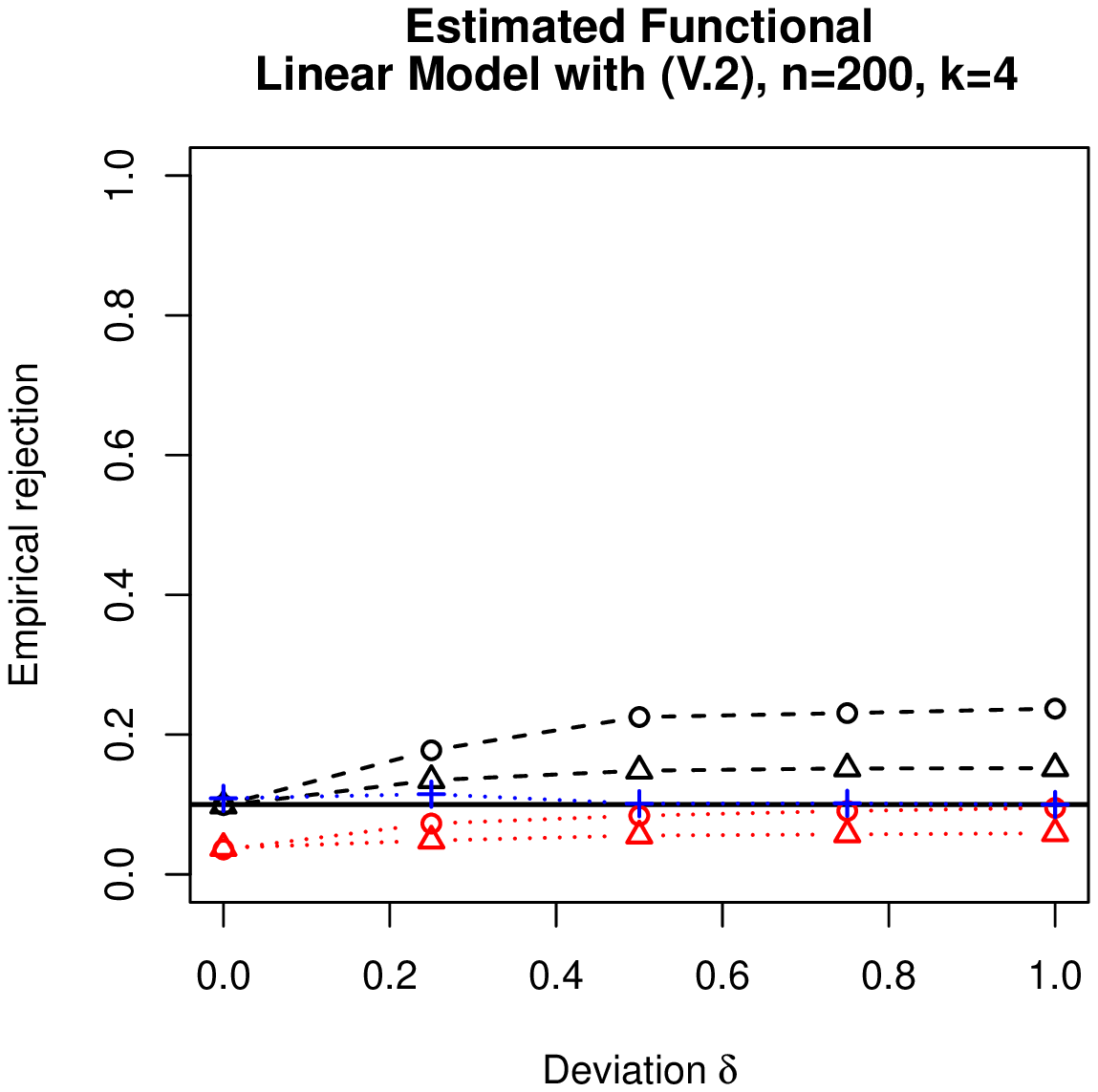}
\par\end{centering}
\hfill\includegraphics[bb=180bp 560bp 280bp 280bp,clip,scale=0.5,angle=270]{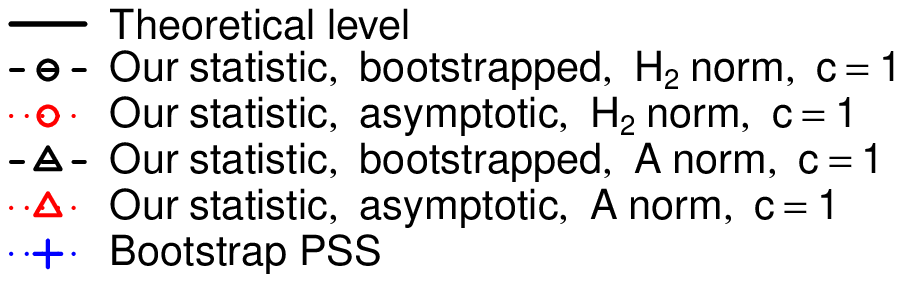}\hfill\hfill
\caption{Empirical rejection for scalar $Y$.
\label{fig:Scalar-Y}}
\end{figure}

\begin{figure}
\begin{centering}
\includegraphics[scale=0.5,angle=270]{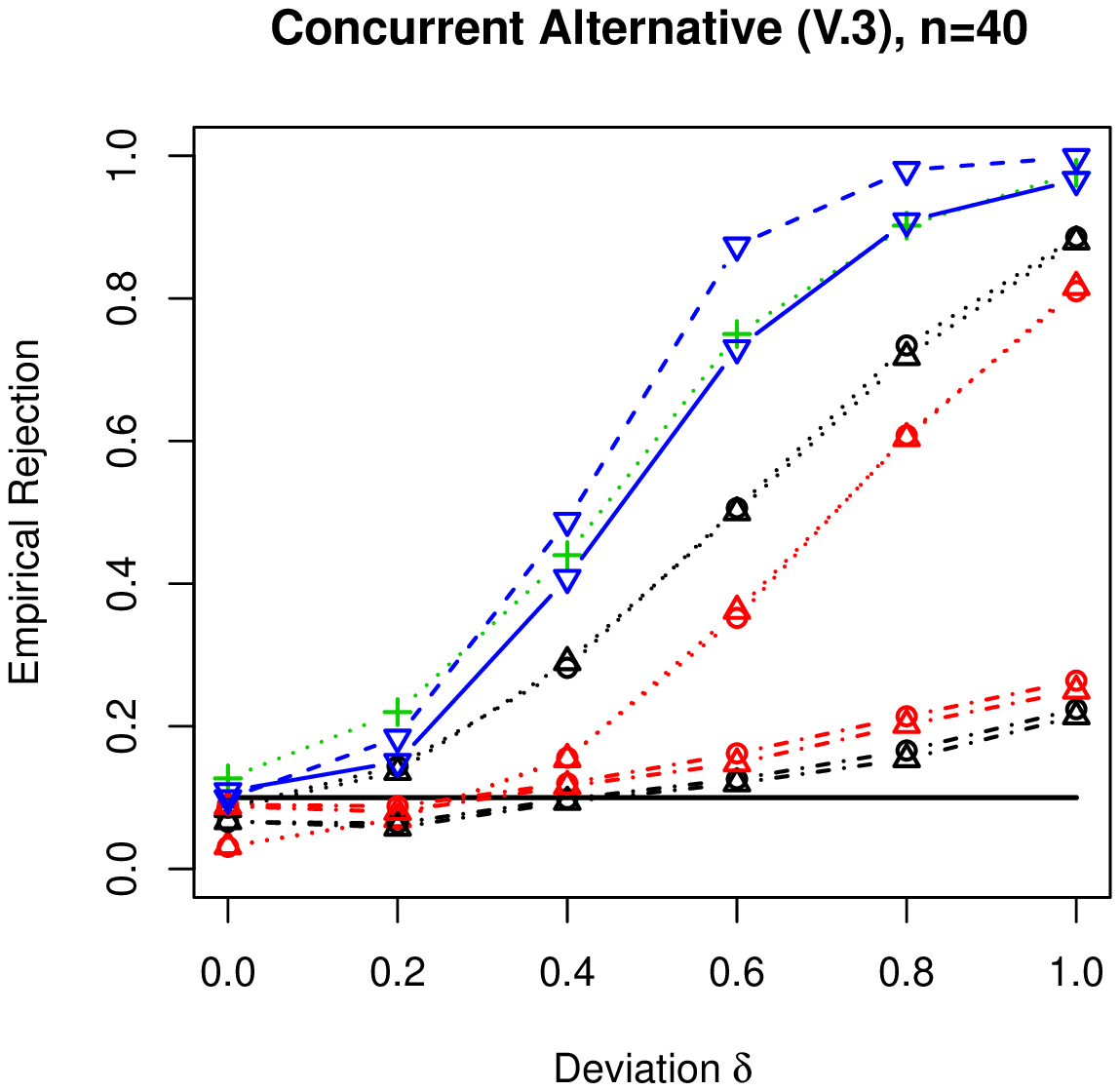}\includegraphics[scale=0.5,angle=270]{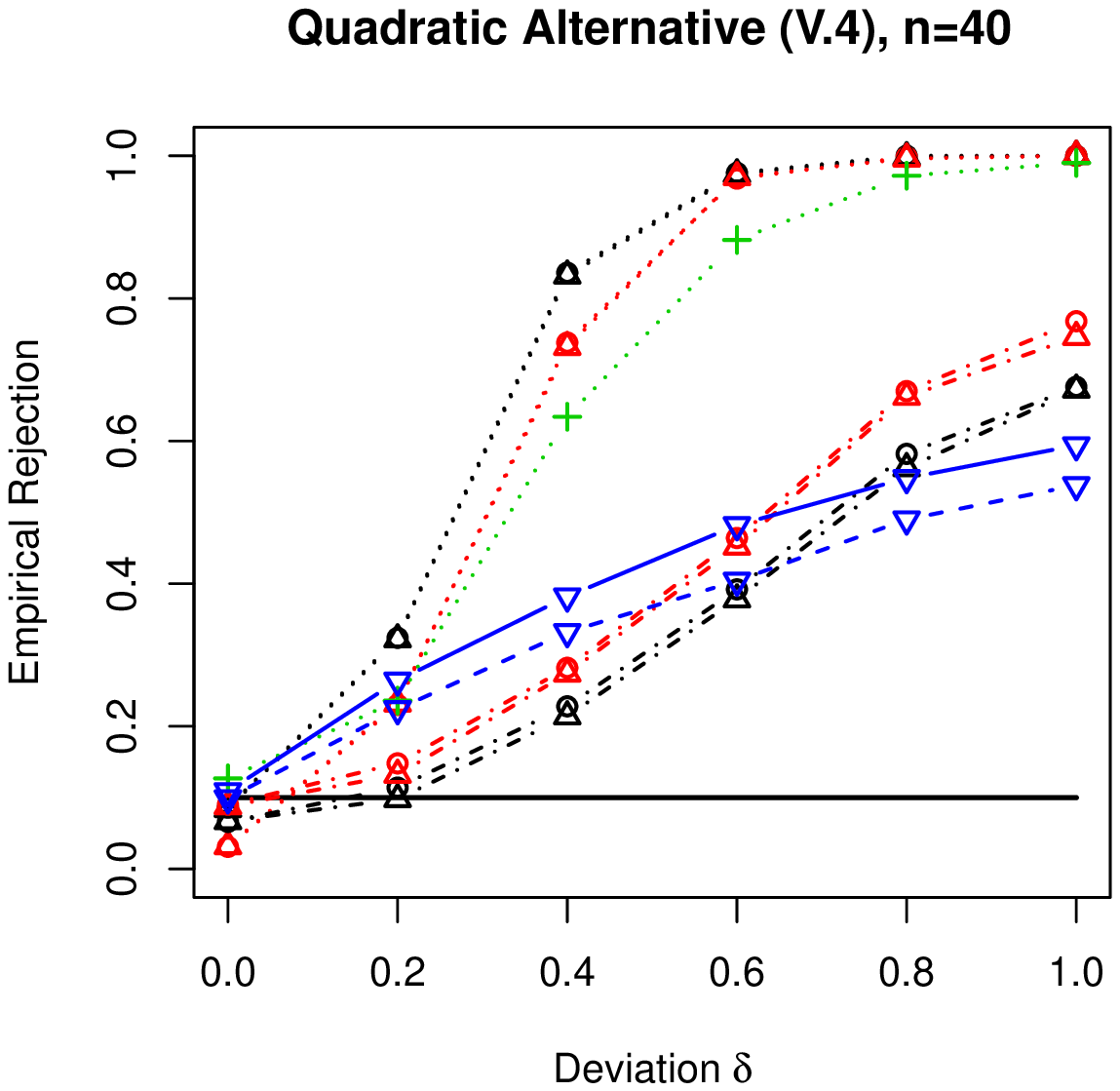}
\par\end{centering}

\begin{centering}
\includegraphics[scale=0.5,angle=270]{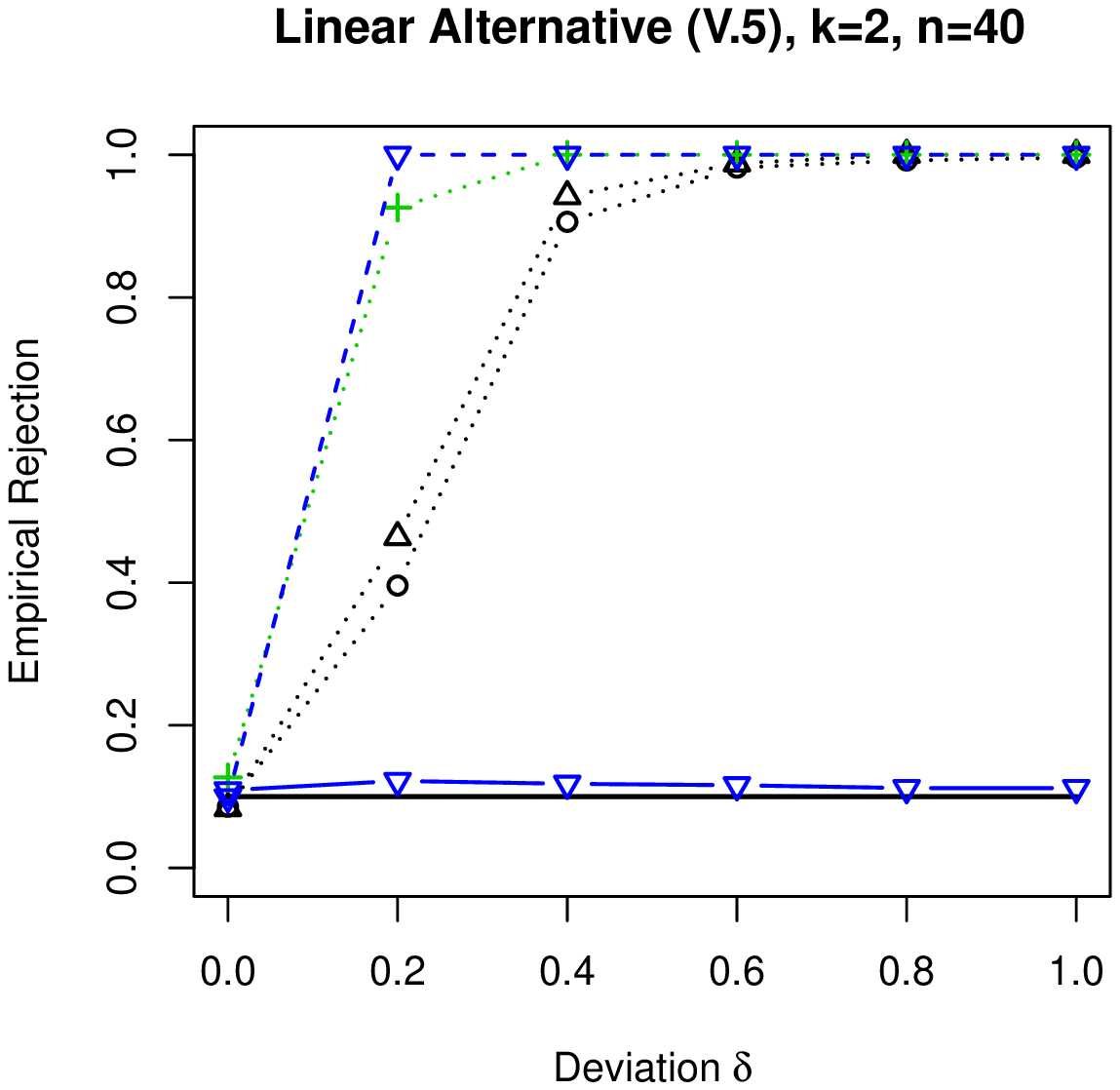}\includegraphics[scale=0.5,angle=270]{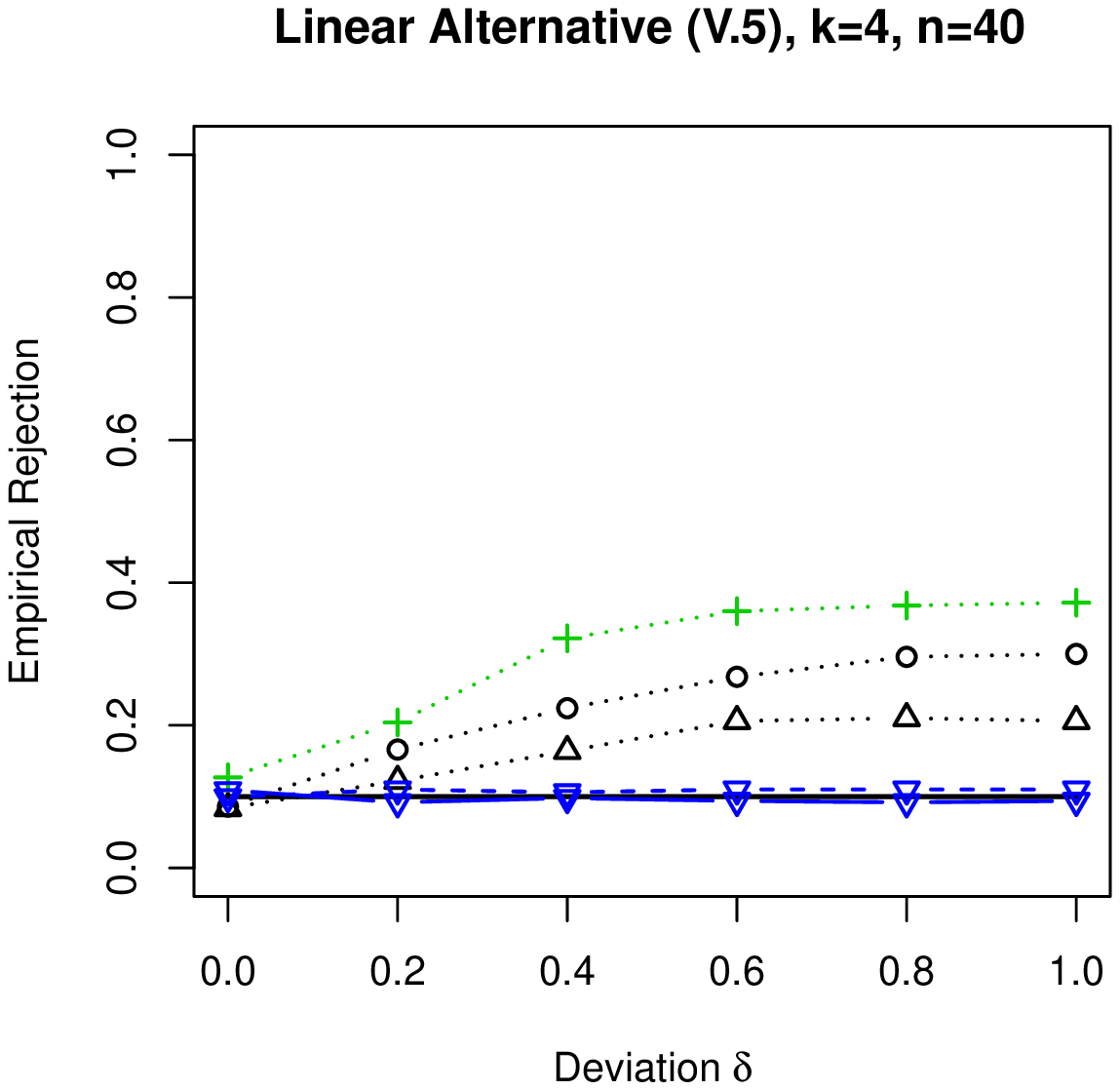}
\par\end{centering}

\begin{centering}
\includegraphics[scale=0.5,angle=270]{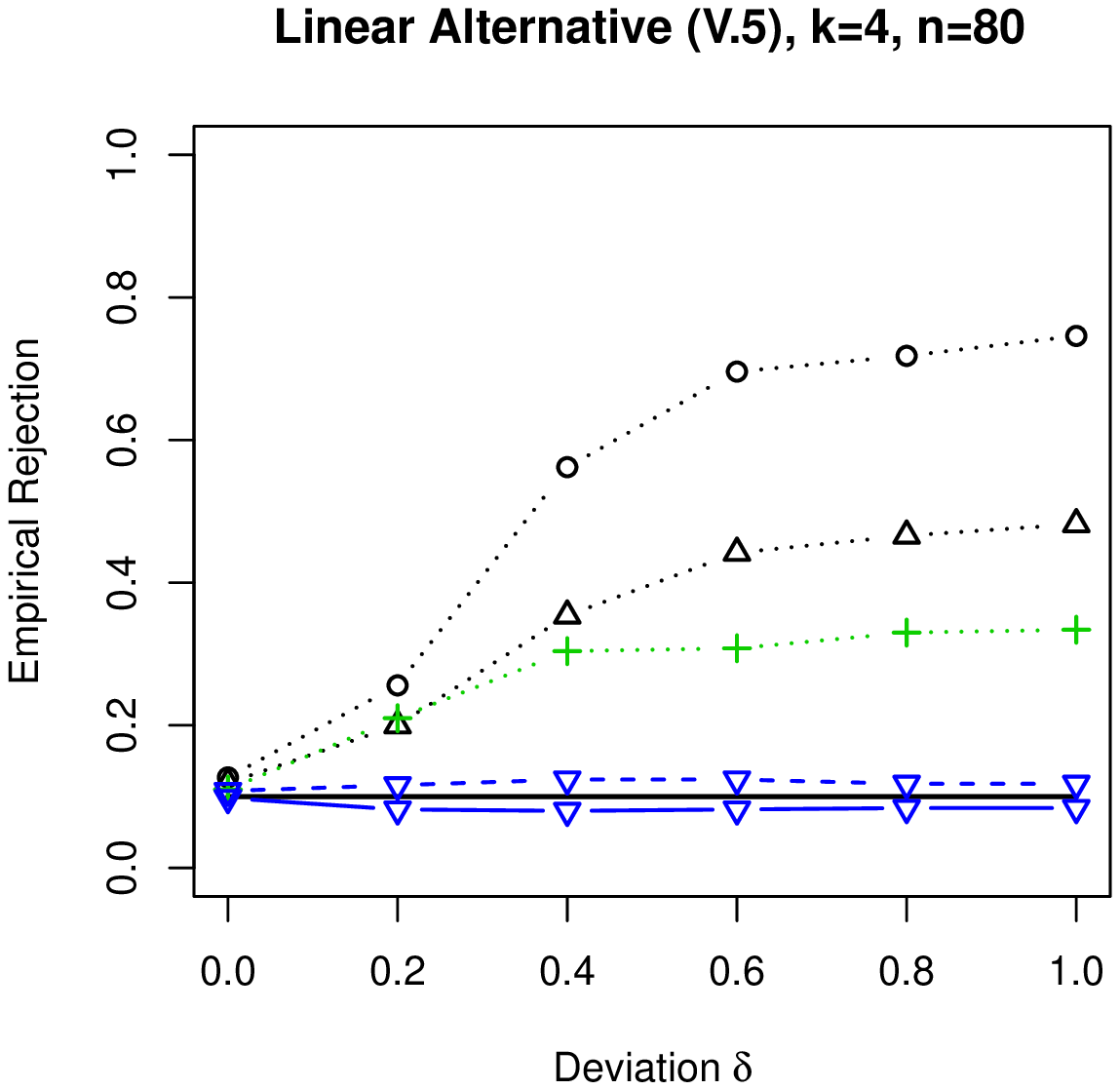}\includegraphics[scale=0.5,angle=270]{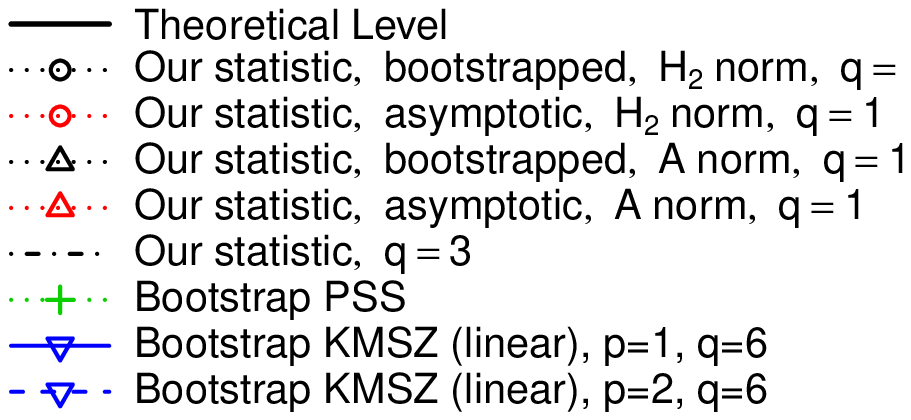}
\par\end{centering}

\caption{Empirical rejection for functional $Y$.\label{fig:Functional-Y}}
\end{figure}

\end{document}